\documentclass[twoside]{amsart}
\usepackage{amsmath,amssymb,amsthm}
\usepackage{verbatim}
\usepackage{paralist}
\usepackage{graphics} 
\usepackage{epsfig} 
\usepackage[colorlinks=true]{hyperref}
\hypersetup{urlcolor=blue, citecolor=red}

\newtheorem{theorem}{Theorem}[section]

\newtheorem{lemma}[theorem]{Lemma}

\theoremstyle{definition}
\newtheorem{definition}[theorem]{Definition}
\newtheorem{remark}{Remark}

\newcommand{\D}{\mathcal{D}}
\newcommand{\E}{\mathcal{E}}
\newcommand{\F}{\mathcal{F}}
\newcommand{\R}{\mathbb{R}}
\newcommand{\N}{\mathbb{N}}
\newcommand{\dd }{\,\mathrm{d}}
\newcommand{\X}{\mathbf{X}}

\newcommand{\EW}{\mathbb{E}}
\newcommand{\vX}{\widehat{\X}}
\newcommand{\vT}{\widehat{T}}

\title{Probabilistic Interpretation of Electrical Impedance Tomography}
\author{Petteri Piiroinen}
\address{Petteri Piiroinen\\ÊDepartment of Mathematics and Statistics\\  University of Helsinki \\ \phantom{Ho}FI-00014 Helsinki, Finland}
\email{petteri.piiroinen@helsinki.fi}
\author{Martin Simon}

\address{Martin Simon\\ÊInstitute of Mathematics\\ Johannes Gutenberg University\\ 55099 Mainz, \phantom{Ho}Germany} 
\email{simon@math.uni-mainz.de}
\date{\today}

\chardef\bslchar=`\\ 

\providecommand{\qedsymbol}{\leavevmode
  \hbox to.77778em{%
  \hfil\vrule
  \vbox to.675em{\hrule width.6em\vfil\hrule}%
  \vrule\hfil}}

\catcode`\|=0
\begingroup \catcode`\>=13 
\gdef\?#1>{{\normalfont$\langle$\textup{#1}$\rangle$}}
\gdef\0{\relax}
\endgroup
\def\<#1>{{\normalfont$\langle$\textup{#1}$\rangle$}}

\def\latex/{{\protect\LaTeX}}

\begin{document}
\maketitle
\markboth{Petteri Piiroinen and Martin Simon}{Probabilistic Interpretation of EIT}

\begin{abstract}
In this paper, we give probabilistic interpretations of both, the forward and the inverse problem of electrical impedance tomography with possibly anisotropic, merely measurable conductivities: 
Using the theory of symmetric Dirichlet spaces, Feynman-Kac type formulae corresponding to different electrode models on bounded Lipschitz domains are derived. Moreover, we give a probabilistic interpretation of the Calder\'{o}n inverse conductivity problem in terms of reflecting diffusion processes and their corresponding boundary trace processes. 
\end{abstract}

\section{Introduction}
In this work we derive purely probabilistic representations in the form of
Feynman-Kac type formulae for solutions of the conductivity equation 
\begin{equation}\label{eqn:con}
\nabla\cdot(\kappa\nabla u)=0
\end{equation}
posed on a bounded, simply connected domain $D\subset\mathbb{R}^d,\
d\geq 2$, with Lipschitz boundary $\partial D$ and possibly anisotropic
uniformly elliptic and uniformly bounded conductivity
subject to different boundary conditions modeling electrode measurements.  
Moreover, we provide a probabilistic interpretation of the inverse conductivity problem 
of electrical impedance tomography (EIT), the so-called \emph{Calder\'{o}n problem} 
which reads as follows: \emph{Given the Cauchy data on the boundary, i.e., 
all pairs of voltage and current measurements on $\partial D$, is it possible 
to determine the conductivity $\kappa$ uniquely?} Our probabilistic interpretation 
generalizes results for the reflecting Brownian motion (RBM) obtained 
by Hsu \cite{Hsu2}.  

Although it is beyond the scope of this paper we would like to emphasize that 
Feynman-Kac type representation formulae provide a versatile tool when it 
comes to problems with random, rapidly oscillating coefficients. For a 
one-dimensional statistical inverse problem, such a setting was recently 
studied by Nolen and Papanicolaou \cite{Nolen}. Moreover, due to the advent 
of multicore computing architectures, Feynman-Kac type representation formulae 
can yield fast and scalable parallel algorithms for stochastic numerics, 
see for instance the recent articles \cite{Bossyetal, Lejay1, LejayMaire,Yan}. 

For non-divergence form operators with smooth coefficients on smooth
bounded domains, it is well-known that reflecting diffusion processes
satisfy Skorohod type stochastic differential equations, see the
celebrated article by Lions and Sznitman \cite{Lions}. The construction
in the case of divergence form operators with merely measurable
coefficients requires the theory of symmetric Dirichlet spaces and is a
major challenge for an arbitrary Euclidean domain $D$ due to the fact
that the underlying Dirichlet space is not necessarily regular on
$L^2(D)$. Thus in general, the reflecting diffusion process can only be
constructed on some abstract closure of $D$, see the pioneering 
work by Chen \cite{Chen}. 
When $D$ is a bounded Lipschitz domain, Bass and Hsu \cite{Bass}
constructed a reflecting Brownian motion living on $\overline{D}$ by
showing that the so-called Martin-Kuramochi boundary coincides with the
Euclidean boundary in this case. A general diffusion process on a
bounded Lipschitz domain, even allowing locally a finite number of
H\"older cusps, was first constructed by Fukushima and Tomisaki
\cite{FukushimaTomisaki}. We use this process here in order to derive
Feynman-Kac type representation formulae for the solutions of 
boundary value problems with Neumann, respectively Robin  
boundary conditions modeling electrode measurements, for the 
conductivity equation (\ref{eqn:con}). However, the construction 
in \cite{FukushimaTomisaki} concentrates on the strong Feller 
resolvent of the process rather than on its transition kernel density. 
Therefore we give a proof of the H\"older continuity up to the boundary 
of the latter.
  
Probabilistic approaches to both, parabolic and elliptic boundary value
problems for second order differential operators have been studied by
many authors, starting with Feynman's Princeton thesis \cite{Feynman}
and the article \cite{Kac} by Kac. The probabilistic approach to the
Dirichlet problem for a general class of second-order elliptic operators
with merely measurable coefficients, even allowing singularities of a
certain type, was elaborated by Chen and Zhang \cite{ChenZhang}; See
also Zhang's paper \cite{Zhang}. However, there are only few works that
treat Feynman-Kac type representation formulae for Neumann or Robin type
boundary conditions.  Moreover the approaches existing in the literature
consider either the Laplacian, see e.g. \cite{Bass,Brosamler,Hsu}, or
non-divergence form operators with smooth coefficients, see e.g.
\cite{Freidlin, Papanicolaou, Pardoux}. 
For the particular case of the conductivity equation (\ref{eqn:con}) on
bounded domains we generalize both, the Feynman-Kac formula
for the Robin problem on domains with boundary of class $C^3$ for an
isotropic $C^{2,\gamma}$-smooth conductivity, $\gamma>0$, obtained by
Papanicolaou \cite{Papanicolaou} as well as the representation
obtained by Bench{\'e}rif-Madani and Pardoux \cite{Pardoux} for the
Neumann problem under similar regularity assumptions. While both of
the aforementioned approaches use stochastic
differential equations and It\^{o} calculus, our approach is based on the 
theory of symmetric Dirichlet spaces, following the pioneering work  
by Bass and Hsu \cite{Bass} for the RBM.
Our Feynman-Kac type formula for the Robin boundary condition 
corresponding to the so-called \emph{complete electrode model} is 
valid for possibly anisotropic uniformly elliptic uniformly bounded conductivities with 
merely measurable components on bounded Lipschitz domains.
For the Neumann boundary condition corresponding to the so-called 
\emph{continuum model} we have to impose a slightly stronger regularity
assumption for the conductivity in some neighborhood of the boundary. 

During the preparation of this work we became aware of the paper
\cite{ChenZhang2} by Chen and Zhang, where a probabilistic approach to
some mixed boundary value problems with singular coefficients is
derived.  In contrast to our setting, however, the mixed boundary
conditions studied there come from a singular lower-order
term of the differential operator.

Finally we would like to point out that our Feynman-Kac type formulae yield 
continuity of the electric potential up to the boundary, a result that
is apparently not so easy to obtain by standard Sobolev regularity theory for
linear elliptic boundary value problems. In fact, by the celebrated 
Wiener criterion, solutions of the Laplace equation with Dirichlet  boundary 
conditions are continuous at a boundary point if and only if the so-called 
Wiener integral associated with this point diverges.
For Robin boundary conditions, on the other hand, the situation was 
not as well understood for quite a long time. In 2001 Griepentrog and Recke were 
able to prove continuity up to the boundary under very general assumptions,
however, using a rather abstract framework based on Sobolev-Campanato spaces, cf.~\cite{Griepentrog}.
In contrast to their (much more general) method, our proof is purely probabilistic 
and rather simple.

The rest of the paper is structured as follows: We start in Section \ref{section2} by 
collecting some preliminaries concerning electrical impedance 
tomography as well as standard Dirichlet space theory for reflecting
diffusion processes. In Section \ref{section3} we show that the transition kernel density 
of the underlying reflecting diffusion process is H\"older continuous 
up to the boundary which enables the refinement of the process.
Subsequently, in Section \ref{section4}, the Feynman-Kac type formulae 
will be derived. Furthermore, a martingale formulation for the complete electrode model 
is given. Then in Section \ref{section5} we provide a probabilistic interpretation
of the Calder\'{o}n problem. Finally, we conclude with a brief summary 
of our results.

\section{Preliminaries}
\label{section2}
First a word about notation:
We denote by $\left<\cdot,\cdot\right>$ the standard inner product on
$L^2(D)$ and by $\lvert\lvert\cdot\rvert\rvert$ the corresponding norm. We use the subscript \lq$\diamond$\rq\ to denote standard Lebesgue, respectively Sobolev spaces with a certain normalization, namely
\begin{equation*}
L^2_{\diamond}(\partial D):=\Big\{\phi\in L^2(\partial D):\int_{\partial D}\phi\dd\sigma(x)=0\Big\},
\end{equation*}
where $\sigma$ denotes the $(d-1)$-dimensional Lebesgue surface measure, and
\begin{equation*}
H^{1}_{\diamond}(D):=\Big\{\phi\in H^{1}(D):\left<\phi,1\right>=0\Big\}.
\end{equation*}
For the reason of notational
compactness we use the Iverson brackets: Let $S$ be a mathematical statement, then 
\begin{equation*}\left[S\right]=\begin{cases}
1,\quad&\text{if }S\text{ is true}\\
0,\quad&\text{otherwise}.
\end{cases}
\end{equation*}
We also use the Iverson brackets $[x\in B]$ to denote the indicator function of a set $B$, which we abbreviate by $[B]$ if there is no danger of confusion.
In what follows, various unimportant constants will be denoted $c,c_1,c_2,...$
and they may vary from line to line. 

\subsection{Modeling of electrode measurements in EIT}
Throughout this paper, let $D$ denote a bounded
Lipschitz domain with connected complement and Lipschitz parameters
$(r_D,c_D)$, i.e., there exist constants $r_D>0$
and $c_D>0$ so that for every $x\in\partial D$ there is a ball
$B(x,r_D)$ such that after rotation and translation $\partial D\cap
B(x,r_D)$ is the graph of a Lipschitz function in the first $d-1$
coordinates with Lipschitz constant no larger than $c_D$ and $D\cap
B(x,r_D)$ lies above the graph of this function. Note that without loss of generality we may take $c_D>1$.
We assume that the, possibly anisotropic, conductivity is defined by a symmetric,
matrix-valued function $\kappa:D\rightarrow \mathbb{R}^{d\times d}$ with
components in $L^{\infty}(D)$ such that $\kappa$ is uniformly bounded and uniformly elliptic, i.e., there exists some constant $c_0>0$ such that
\begin{equation}\label{eqn:ellipticity}
c_0^{-1}\lvert\lvert\xi\rvert\rvert^2\leq \xi\cdot\kappa (x)\xi\leq
c_0\lvert\lvert\xi\rvert\rvert^2,\quad \text{for every
}\xi\in\mathbb{R}^d\text{ and a.e. }x\in D.  
\end{equation}

Moreover we will explicitly state if we use one of the following assumptions:
\begin{enumerate}
\item[(A1)]
There exists a neighborhood $\mathcal{U}$ of the boundary $\partial D$
such that $\kappa|_{\mathcal{U}}$ is isotropic and equal to $1$. 
\item[(A2)] There exists a finite collection $\Gamma=\{\Gamma_j,0\leq
j\leq M\}$ of $C^{1,1}$ surfaces that divide $D$ into disjoint open
sub-domains $\{\mathcal{U}_j,0\leq j\leq M\}$ so that $\partial D\subset
\partial\mathcal{U}_0$ and $\kappa|_{\mathcal{U}_j}\in
W^{1,\infty}(\mathcal{U}_j)$, $j=1,...,M$.  
\end{enumerate}
We are going to require assumption (A1) for the probabilistic interpretation of the
Neumann boundary value problem corresponding to the continuum model in Section \ref{section4}
while assumption (A2) will only be used for the probabilistic interpretation 
of the inverse conductivity problem in Section \ref{section5}. In particular 
the derivation of the Feynman-Kac type formula for the Robin boundary 
condition corresponding to the complete electrode model does not 
require any of these additional regularity assumptions.

\begin{remark}
Notice that the assumption (A1) from above is not very restrictive as it can
be shown using extension techniques that for domains $\hat{D}, D\subset
\R^d$ such that $D\subset\hat{D}$, the knowledge of both, the
Dirichlet-to-Neumann map $\Lambda_{\kappa}$ on $\partial D$ and
$\kappa|_{\hat{D}\backslash\overline{D}}$ yields the
Dirichlet-to-Neumann map $\hat{\Lambda}_{\kappa}$ on $\partial \hat{D}$. 
\end{remark}

The forward problem of EIT can be modeled by different measurement
models. In the so-called \emph{continuum model}, one assumes that it
is possible to measure the electric potential $u$ on the whole boundary
for a prescribed conormal flux through $\partial D$
\begin{equation}\label{eqn:continuum}
\partial_{\kappa\nu}u:=\kappa\nu\cdot \nabla u|_{\partial D}=f,
\end{equation}
where $\nu$ denotes the exterior unit normal vector on $\partial D$ and
$f\in L^2_{\diamond}(\partial D)$ a bounded function modeling the 
signed density of the outgoing current.

The most accurate forward model for real-life impedance tomography is
the \emph{complete electrode model}, cf.~\cite{Cheney,Somersalo},
where under the assumption that measurements are performed using $N$
electrodes $E_l,\ l=1,...,N,$ on the boundary $\partial D$, the electric
potential $u$ satisfies the Robin boundary condition 
\begin{equation}\label{eqn:cem}
\begin{aligned}
&\kappa\nu\cdot\nabla u|_{\partial D}+gu|_{\partial D}=f\quad\text{on
}\partial D, 
\end{aligned}
\end{equation}
for piecewise constant functions $g,f:\partial D\rightarrow\mathbb{R}$
  given by 
\begin{equation}\label{eqn:functions}
g:=\frac{1}{z}\sum_{l=1}^N[E_{l}],\quad
f:=\frac{1}{z}\sum_{l=1}^N U_l[E_{l}].
\end{equation}
Here, $[E_l]$ is the indicator function of the $l$-th electrode
and $U=(U_1,...,U_N)^T$ denotes the prescribed voltage pattern. 
The positive constant $z\in\mathbb{R}_+$ is the so-called
\emph{contact impedance} which models electrochemical effects at the
electrode-object interface. The electrodes in the complete electrode 
model $E_l\subset\partial D,\ l=1,...,N$, are modeled by disjoint, 
simply connected, surface patches, each having a smooth boundary 
curve. Moreover we always assume that the ground voltage has been 
chosen such that 
\begin{equation}\label{ground}
\sum_{l=1}^N U_l =0.
\end{equation}
For given voltage pattern $U\in\mathbb{R}^N$ satisfying (\ref{ground}),
the equations (\ref{eqn:con}) and (\ref{eqn:cem}) define the electric
potential $u\in H^1(D)$ uniquely, cf.~\cite{Somersalo}.

\subsection{Preliminaries from symmetric Dirichlet space theory}
In his seminal paper \cite{Fukushima}, Fukushima established a
one-to-one correspondence between regular symmetric Dirichlet spaces and
symmetric Hunt processes. We assume that the reader is familiar with the
basic results from the theory of symmetric Dirichlet spaces, as
elaborated for instance in the monograph \cite{Fukushimaetal}. 
Let us consider the following symmetric bilinear forms on $L^2(D)$:
\begin{equation}\label{eqn:Dirichlet}
\E(v,w)=\int_{D}\kappa\nabla v(x)\cdot\nabla w(x)\dd  x,\quad
v,w\in\D(\E)=H^1(D);
\end{equation}
\begin{equation}\label{eqn:Dirichlet1}
\E_0(v,w)=\int_{D}\nabla v(x)\cdot\nabla w(x)\dd  x,\quad
v,w\in\D(\E_0)=H^1(D).
\end{equation} 
It is well-known that (\ref{eqn:Dirichlet1}) is associated with the
reflecting Brownian motion (scaled by a factor $2$), while
(\ref{eqn:Dirichlet}) corresponds to a general reflecting diffusion
process. First recall that we may associate with the Dirichlet space
$(\D(\E),\E)$ a non-positive definite self-adjoint operator
$(\mathcal{L},\D(\mathcal{L}))$ such that for $v\in\D(\mathcal{L})$ we
have $\left<-\mathcal{L}v,w\right>=\E(v,w)$ for all $w\in\D(\E)$ and  
\begin{equation*}
\D(\mathcal{L})=\Big\{v\in\D(\E):\exists \phi\in L^2(D)\text{ s.t. }
\E(v,w)=\int_D\phi w\dd  x\ \forall w\in \D(\E)\Big\}.
\end{equation*}
This Dirichlet space is regular on $L^2(D)$, i.e., $\D(\E)\cap
C(\overline{D})$ is dense in both,
$(C(\overline{D}),\lvert\lvert\cdot\rvert\rvert_{\infty})$ and
$(\D(\E),\lvert\lvert\cdot\rvert\rvert_{\E_1})$, where $\lvert\lvert\cdot\rvert\rvert_{\E_{\beta}}:=\sqrt{\E_{\beta}(\cdot,\cdot)}$, $\beta>0$, with $\E_{\beta}(\cdot,\cdot):=\E(\cdot,\cdot)+\beta\left<\cdot,\cdot\right>$.
This follows directly from the convergence of  $v_{\alpha}:=\alpha
G_{\alpha}v$ in $(\D(\E),\lvert\lvert\cdot\rvert\rvert_{\E_1})$ as $\alpha\rightarrow\infty$,
where $\{G_{\alpha},\alpha>0\}$ denotes the strongly continuous
resolvent on $L^2(D)$ associated with $(\D(\E),\E)$, cf.~\cite{Fukushimaetal}, if one has that $v_{\alpha}\in
C(\overline{D})$ for all $\alpha>0$. This is, for instance, a
consequence of our Theorem \ref{thm:1}.
Moreover the Dirichlet space $(\D(\E),\E)$ is local in the sense that
$\E(v,w)=0$, whenever $\mathrm{supp}(v)$ and $\mathrm{supp}(w)$ are
disjoint compact sets. The \emph{capacity} of an open subset $O$ of
$\overline{D}$ is defined by
$\mathrm{Cap}(O)=\inf_{v\in\D(\E)}\{\E_1(v,v):v\geq 1\text{ a.e. on }O\}$ and that of a
general subset is given by $\mathrm{Cap}(B)=\inf\{\mathrm{Cap}(O):O
\text{ is open and }B\subset O\}$. $\mathrm{Cap}$ is a Choquet capacity, cf.~\cite{Fukushimaetal},
and a Borel set $B\subset\overline{D}$ is called $\E$-exceptional if
$\mathrm{Cap}(B)=0$.
The general theory of regular local symmetric Dirichlet spaces yields
that there exist an $\E$-exceptional set
$\mathcal{N}\subset\overline{D}$ and a conservative diffusion process
$\X=(\Omega,\F,\{\X_t,t\geq 0\},\mathbb{P}_{x})$, starting from every
$x\in\overline{D}\backslash \mathcal{N}$ (denoted \lq for quasi every
(abbreviated q.e.) $x\in\overline{D}$\rq), properly associated with $(\D(\E),\E)$.
That is, for every non-negative $\phi\in L^2(D)$ the transition
semigroup $P_t\phi(x):=\EW_{x}\phi(\X_t)$,
$x\in\overline{D}\backslash\mathcal{N}$, of $\X$ is a version of the strongly continuous
semigroup $T_t\phi$ of contractions on $L^2(D)$ associated with
$(\D(\E),\E)$. Note that $\X$ is
in general not a semimartingale and that it is not known in general
where the $\E$-exceptional set $\mathcal{N}$ is located. The latter
imposes a severe limitation on practical applications which we have to
remove here.

For convenience of the reader let us recall the definition of additive
functionals of Markov processes depending on the potential theory of the
bilinear form $\E$ from \cite{Fukushimaetal}. Let $\{\F_t,t\geq 0\}$
denote the minimal augmented filtration generated by $\X$ and without
loss of generality let us assume that $\X$ is defined on the canonical
sample space $\Omega=C([0,\infty);\overline{D})$ on which the time shift
operator $\Theta$  is well-defined by
$\mathbf{X}_s(\Theta_t(\omega))=\mathbf{X}_{t+s}(\omega)$, $s,t\geq 0$.
\begin{definition}
A real-valued stochastic process $A=(\Omega,\F, \{A_t,t\geq
0\},\mathbb{P}_{x})$ is an \emph{additive functional} (abbreviated
AF) of $\mathbf{X}$ if the following conditions hold:
\begin{enumerate}
\item[(i)] $A_t$ is adapted to $\F_t$ for every $t\geq 0$;
\item[(ii)] There exists a defining set $\Lambda\in\F_{\infty}$ and an
$\E$-exceptional set $\mathcal{N}\subset\overline{D}$ such that
$\mathbb{P}_{x}(\Lambda)=1$ for $x\in\overline{D}\backslash
\mathcal{N}$ and $\Theta_t\Lambda\subset\Lambda$ for every $t\geq 0$;
\item[(iii)] For $\omega\in\Lambda$, $A_t(\omega)$ is right-continuous
and has left limits in $t\in[0,\infty)$ with $A_0(\omega)=0$ and $\lvert
A_t(\omega)\rvert<\infty$ for every $t<\infty$;
\item[(iv)] $A_{t+s}(\omega)=A_t(\omega)+A_s(\Theta_t\omega)$.
\end{enumerate}
If in addition the mapping $t\mapsto A_t(\omega)$ is positive and
continuous for each $\omega\in\Lambda$, then $A$ is a \emph{positive
continuous additive functional} (abbreviated PCAF).
An additive functional admitting a defining set $\Lambda$ with
$\mathbb{P}_{x}(\Lambda)=1$ for all $x\in\overline{D}$ is called an
\emph{additive functional in the strict sense.}
\end{definition}
\begin{definition}
A positive Borel measure $\lambda$ on $\overline{D}$ is called a
\emph{smooth measure} of $\mathbf{X}$ if the following conditions hold:
\begin{enumerate}
\item[(i)] $\lambda$ charges no sets of zero capacity, i.e.,
$\lambda(\mathcal{N})=0$ if $\mathcal{N}$ is $\E$-exceptional;
\item[(ii)] There exists an increasing sequence $\{C_k\}_{k\in\N}$ of
closed sets satisfying $$\lim_{k\rightarrow\infty}\lambda(K\backslash
C_k)=0 \text{ for every compact set }K,$$ such that
$\lambda(C_k)<\infty$ and
$\lambda(\overline{D}\backslash\cup_{k\in\N}C_k)=0$. 
\end{enumerate}
\end{definition}
We can now formulate the well-known \emph{Revuz correspondence}. The
family $\mathcal{A}_c^+$ of all PCAFs of $\mathbf{X}$ and the family
$\mathcal{S}$ of all smooth measures on $\overline{D}$ are in one-to-one
correspondence. In other words, for every $A \in
\mathcal\mathcal{A}_c^+$ there exists a unique $\lambda
\in \mathcal {S}$, and vice versa, that satisfy
\begin{equation}\label{eqn:Revuz}
\lim_{t\rightarrow
0+}\frac{1}{t}\int_{D}\EW_{x}\Big\{\int_0^t\phi(\mathbf{X}_s)\dd 
A_s\Big\}\psi(x)\dd 
x=\int_{\overline{D}}\phi(x)\psi(x)\dd \lambda(x),
\end{equation}
for every non-negative Borel measurable function $\phi$ and
$\gamma$-excessive function $\psi$, $\gamma\geq 0$. Recall that a non-negative
function $\psi$ is called $\gamma$-excessive (with respect to $\X$) if
$\lim_{t\rightarrow 0+}e^{-\gamma t}\EW_{x}\psi(\X_t)= \psi(x)$ for
$x\in\overline{D}$.  By the Lipschitz property of $\partial D$, the
Lebesgue surface measure $\sigma$ on $\partial D$ exists and
it is easy to see that $\sigma$ is a smooth measure of $\X$. Let thus $L$
denote the PCAF of $\mathbf{X}$ whose Revuz measure is given by $\sigma$. 
In analogy to the notion for case of smooth coefficients, see, e.g.,
\cite{Hsu,Papanicolaou,Pardoux}, we call $L$ the \emph{boundary local
time} of the reflecting diffusion process $\mathbf{X}$.

Let us conclude this section by recalling that in the framework of
symmetric Dirichlet spaces, the celebrated Fukushima decomposition and
the corresponding transformation formula, cf.~\cite{Fukushimaetal}, play
in some sense the roles of the Doob-Meyer decomposition and It\^{o}'s
formula for semimartingales:
If $v\in\D(\E)$, then the composite process
$v(\X)=(\Omega,\F,\{v(\X_t),t\geq 0\},\mathbb{P}_{x})$ admits the
following unique decomposition
\begin{equation}\label{eqn:Fukudecomp}
\tilde{v}(\X_t)=\tilde{v}(\X_0)+M_t^{[v]}+N_t^{[v]},\quad\text{for all }
t> 0,\quad \mathbb{P}_{x}\text{-a.s.,}
\end{equation} 
holding for q.e. $x\in\overline{D}$, where $\tilde{v}$ is a
quasi-continuous version of $v$, $M^{[v]}$ is a martingale AF of $\X$
having finite energy and $N^{[v]}$ is a continuous AF of $\X$ having
zero energy. A function $\phi$ which is defined q.e. on $\overline{D}$
is called \emph{quasi-continuous} if for every $\varepsilon>0$ there
is an open subset $O\subset\overline{D}$ with
$\mathrm{Cap}(O)<\varepsilon$ such that $\phi|_{\overline{D}\backslash
O}$ is continuous. It is well-known, that every $v\in\D(\E)$ has a
quasi-continuous version. Recall moreover that the \emph{energy }of an
AF $A$ is defined as
\begin{equation*}
\lim_{t\rightarrow 0+}\frac{1}{2t}\int_{D}\EW_{x}A_t^2\dd  x
\end{equation*}
and that the elements of the set \begin{equation*}\begin{split}\{&A: A\text{ is an AF of }\mathbf{X}
\text{ with }\E\text{-exceptional set }\mathcal{N} \text{ s.t. }\\
& \EW_{x}A_t^2<\infty \text{ for all }t>0
\text{ and
  }\EW_{x}A_t=0\text{ for all }x\in\overline{D}\backslash
  \mathcal{N} \}\end{split}\end{equation*} are called \emph{martingale AFs.}
 
\section{Refinement of the reflecting diffusion process}\label{section3}
In order to refine the diffusion process $\X$ to start from every
$x\in\overline{D}$ without exceptional set we need the well-known
connection between the strongly continuous sub-Markovian semigroup
$\{T_t,t\geq 0\}$ on $L^2(D)$ and the evolution system corresponding to
$(\mathcal{L},\mathcal{D}(\mathcal{L}))$, see, e.g., the monograph
\cite{Pazy}. Namely for every $v_0\in L^2(D)$, the function
$v(t):=T_tv_0$ belongs to the function space $W(0,T;H^1(D),H^{-1}(D))$ 
given by the set
$$\{\phi\in L^2((0,T);H^1(D)):\dot{\phi}\in
L^2((0,t);H^{-1}(D))\}$$ and is the unique solution of the abstract
Cauchy problem 
\begin{equation}\begin{aligned}\label{parabolic}
&\dot{v} + \mathcal{L} v=0 &\ &\text{in }(0,T)\\
&v(0)=v_0.&\ &
\end{aligned}\end{equation}
On the other hand, given (\ref{parabolic}) it is not difficult to verify
that $v$ also satisfies the parabolic equation
\begin{equation}\label{weak}
-\int_0^T\left<v(t),w\right>\dot{\varphi}(t)\dd 
t+\int_0^T\left<\mathcal{L}v(t),w\right>\varphi(t)\dd 
t-\left<v_0,w\right>\varphi(0)=0
\end{equation}
for all $w\in H^1(D)$ and all $\varphi\in C^{\infty}_c([0,T))$.
Moreover it is well-known that $T_t$ is a bounded operator from $L^1(D)$
to $L^{\infty}(D)$ for every $t>0$. Therefore by the Dunford-Pettis
Theorem, it can be represented as an integral operator for every $t>0$,
\begin{equation}
T_t\phi(x)=\int_D p(t,x,y)\phi(y)\dd  y\quad\text{a.e. on }D,\text{ for
every } \phi\in L^1(D),
\end{equation}
where for all $t> 0$ we have $p(t,x,y)\in L^{\infty}(D\times D)$ and
$p(t,x,y)\geq 0$ a.e.. 

The following Theorem generalizes a well-known result for
diffusion processes on $\R^d$ from \cite{Stroock0}.
\begin{theorem}\label{thm:1}
For each fixed $0<t_0\leq T$ there exist positive constants $c_1$
and $c_2$ such that 
\begin{equation}\label{eqn:Hold}
\lvert p(t_2,x_2,y_2)-p(t_1,x_1,y_1)\rvert\leq
c_1(\sqrt{t_2-t_1}+\lvert x_2-x_1\rvert+\lvert y_2-y_1\rvert)^{c_2}
\end{equation}
for all $t_0\leq t_1\leq t_2\leq T$ and all
  $(x_1,y_1),(x_2,y_2)\in\overline{D}\times\overline{D}$.
\end{theorem}
\begin{proof}
The main idea of the proof is the following extension by reflection technique, see for example \cite[Section 2.4.3]{Troianiello}:  
We extend the solution of a parabolic problem by reflection at the boundary, then show that 
this extension again solves a parabolic problem so that we can apply a well-known 
interior regularity result. 

First note that Nash's inequality holds for the underlying
Dirichlet space $(\D(\E),\E)$, i.e., there exists a constant $c_1>0$ such
that 
\begin{equation*}
\lvert\lvert v\rvert\rvert^{2+4/d} \leq c_1 (\E(v,v)+\lvert\lvert
v\rvert\rvert^2)\lvert\lvert v\rvert\rvert_{L^1(D)}^{4/d}\quad\text{for
all }v\in H^1(D).
\end{equation*}
This is a direct consequence of the uniform ellipticity
(\ref{eqn:ellipticity}) and \cite[Corollary 2.2]{Bass}, where Nash's
inequality is shown to hold for the Dirichlet space $(H^1(D),\E_0)$ 
with $D$ a bounded Lipschitz domain.
Analogously to the proof of \cite[Theorem 3.1]{Bass}, it follows thus
from \cite[Theorem 3.25]{Carlen} that the
transition kernel density satisfies an Aronson type Gaussian upper bound
\begin{equation}\label{eqn:Gaussian}
p(t,x,y)\leq c_1
t^{-d/2}\exp\Big(-\frac{\lvert x-y\rvert^2}{c_2t}\Big)\quad\text{for
all }t\leq 1\text{ and }(x,y)\in\overline{D}\times\overline{D}.
\end{equation}
In particular $\sup_{0< t\leq 1}\lvert\lvert
p(t,\cdot,\cdot)\rvert\rvert_{\infty}$ is finite and by Nash's interior H\"older
continuity Theorem, cf.~\cite{Nash}, the estimate (\ref{eqn:Hold}) is true for
every $(x_1,y_1)$, $(x_2,y_2)$ with $d(x_i,\partial
D)$, $d(y_i,\partial D)>c_3$, $i=1,2$, for some constant $c_3>0$ and all $t_0\leq t_1\leq
t_2\leq 1$. Note that by the Markov property of the semigroup the
Chapman-Kolmogorov equation holds, i.e.,
$$p(t_1+t_2,x,y)=\int_D p(t_1,x,z)p(t_2,z,y)\dd  z$$ for every $t_1,
t_2$ and a.e. $x,y\in\overline{D}$, in particular for fixed
$y\in\overline{D}$ the function $v:=p(\cdot,\cdot,y)$ is the unique
solution to (\ref{parabolic}) with initial value $v_0:=p(0,\cdot,y)\in
L^2(D)$. 
Now let $z\in\partial D$ so that by the Lipschitz property of $\partial D$ we
have after translation and rotation
$B(z,r_D)\cap\overline{D}=\{(\tilde{x},x_d)\in
B(z,r_D):x_d>\gamma(\tilde{x})\}$ and $B(z,r_D)\cap\partial
D=\{\tilde{x}\in B(z,r_D):x_d=\gamma(\tilde{x})\}$, where we have
introduced the notation $\tilde{x}=(x_1,...,x_{d-1})^T$. 
Let us furthermore introduce the one-to-one transformation
$\Psi(x):=(\tilde{x},x_d-\gamma(\tilde{x}))$ which maps
$B(z,r_D)\cap\overline{D}$ into a subset of the hyperplane
$\{(\tilde{y},0)\}$ and straightens the boundary $B(z,r_D)\cap\partial
D$. $\Psi$ is a bi-Lipschitz transformation and the Jacobians of both
$\Psi$ and $\Psi^{-1}$ are bounded with bounds that depend only on the
Lipschitz constant $c_D$.
Since $v$ is the solution of (\ref{parabolic}) with appropriate initial
condition, the function $\hat{v}:=v(\cdot,\Psi^{-1}(\cdot))$ must
satisfy the following parabolic equation in
$\hat{D}(z,r_D):=\Psi(B(z,r_D)\cap\overline{D})$:
\begin{eqnarray*}
\int_0^T\dot{\varphi}(t)\int_{\hat{D}(z,r_D)}\hat{v}(t)w\dd  x\dd 
t&=&-\sum\limits_{i,j=1}^d\int_0^T\varphi(t)\int_{\hat{D}(z,r_D)}
\hat{\kappa}_{ij}\partial_i
\hat{v}(t)\partial_j w\dd x\dd  t\\
&&-\varphi(0)\int_{\hat{D}(z,r_D)}\hat{v}_0w\dd x
\end{eqnarray*}
for all $w\in C^{\infty}_c(\hat{D}(z,r_D))$ and all $\varphi\in
C^{\infty}_c([0,T))$. The coefficient $\hat{\kappa}$ is obtained via
change of variables and it is bounded and uniformly elliptic by the
boundedness of the Jacobians of $\Psi$ and $\Psi^{-1}$, respectively.
Now we use reflection on the hyperplane $\{(\tilde{y},0)\}$ via the
mapping $\rho(x):=(\tilde{x},-x_d)$ which yields that the function
$\hat{v}(\cdot,\rho(\cdot))$ satisfies a parabolic equation on
$\rho(\hat{D}(z,r_D))$. Summing up both parabolic equations on $
\hat{D}(z,r_D)$ and on $\rho( \hat{D}(z,r_D))$, respectively, we obtain
that the function 
\begin{equation*}
\check{v}(t,x):=\begin{cases}
\hat{v}(t,x),\quad&x\in \hat{D}(z,r_D)\\
\hat{v}(t,\rho(x)),\quad&x\in \rho(\hat{D}(z,r_D))
\end{cases}
\end{equation*}
satisfies a parabolic equation in $(\hat{D}(z,r_D)\cup\rho(\hat{D}(z,r_D))$.
By Nash's interior H\"older estimate
for $\check{v}$, together with the fact that we may choose $c_3=\frac{r_D}{4c_D}$, we obtain thus
\begin{equation*}
\lvert p(t_2,x_2,\Psi^{-1}(y_2))-p(t_1,x_1,\Psi^{-1}(y_1))\rvert\leq
c_1(\sqrt{t_2-t_1}+\lvert y_2-y_1\rvert)^{c_2}
\end{equation*} 
for all $t_0\leq t_1\leq t_2\leq 1$ and $y_1,y_2\in
\{(\tilde{x},x_d):\lvert\tilde{x}\rvert<c_3,\ x_d\in
(0,r_D/4)\}$. As $\Psi$ is bi-Lipschitz, for fixed $x$, the mapping
$(t,y)\mapsto p(t,x,y)$ is H\"older continuous in
$(t_0,1]\times(B(z,c_3)\cap\overline{D})$ and by symmetry of the
transition kernel density the same holds true for the mapping
$(t,x)\mapsto p(t,x,y)$ for fixed $y$.  
Finally, the assertion on
$(t_0,1]\times\overline{D}\times\overline{D}$ follows due to compactness of $\partial D$ 
and its generalization to arbitrary $T>0$ is obtained after
repeatedly applying Chapman-Kolmogorov.  
\end{proof} 
By \cite[Theorem 2]{Fukushima2} the H\"older continuity of the
transition density kernel ensures that we may refine the process $\X$ to
start from every $x\in\overline{D}$ by identifying the strongly
continuous semigroup $\{T_t,t\geq 0\}$ with the transition semigroup
$\{P_t,t\geq 0\}$. Moreover the decomposition (\ref{eqn:Fukudecomp})
holds $\mathbb{P}_{x}$-a.s. for every $x\in\overline{D}$ if $v$ is
continuous and locally in $H^1(D)$ and the energy measure of $M^{[v]}$
given by $$\dd  \mu_{\left<
v\right>}(x):=2\sum_{i,j=1}^d\kappa_{ij}(x)\partial_i
v(x)\partial_jv(x)\dd  x$$ is a smooth measure of $\X$ in the strict sense, that is,
there is an increasing sequence of finely open sets $\{D_k,k\geq 1\}$ so
that $\cup_{k\in\N}D_k=\overline{D}$, $[D_k]\mu_{\left<v\right>}$ is a
finite Borel measure and the $1$-resolvent $G_1([D_k]\mu_{\left<v\right>})$
is bounded for every $k\geq 1$. In this case, both $M^{[v]}$ and
$N^{[v]}$ can be taken to be additive functionals of $\X$ in the strict
sense.

\section{Feynman-Kac type representation formulae}\label{section4}
Let us first state some auxiliary Lemmata. 

\begin{lemma}\label{lem:exponential}
The transition density kernel $p$ approaches the stationary distribution
uniformly and exponentially fast, that is, there exists a $t_0>0$ and a
constant $c_3>0$ such that for all
$(x,y)\in\overline{D}\times\overline{D}$ and every $t\geq t_0$,
\begin{equation}\label{eqn:stationary}
\lvert p(t,x,y)-{\lvert D\rvert}^{-1}\rvert\leq \exp(-c_3t).
\end{equation}
\end{lemma}
\begin{proof}
It is well-known that there exists an orthonormal basis 
$\{\phi_j:j\in\mathbb{N}\}$ of $L^2(D)$ and an increasing sequence of 
constants $(\lambda_j)_{j\in\N}$ such that $\lambda_0 = 0$, 
$\lambda_1 > 0$ and the functions $\phi_j$ are the weak solutions
of the eigenvalue problem
  \[
  \begin{cases}
    - \nabla \cdot (\kappa \nabla \phi_j) = \lambda_j \phi_j &
    \text{in }D,\\
    \kappa\nu \cdot \nabla \phi_j = 0 & \text{on }\partial D.\\
  \end{cases}
  \]
Using Theorem \ref{thm:1} it is not difficult to see 
that the eigenfunctions satisfy $\phi_j \in
 C(\overline D)$ and for every $t > 0$
 and $x,y \in \overline D$ we have
 \[
 p(t,x,y)=\lvert D\rvert^{-1}+\sum_{j=1}^{\infty} e^{-\lambda_j t}
 \phi_j(x)\phi_j(y)
 \]
by the fact that for every $v_0 \in L^2(D)$, the function $v(t)$ 
given by
  \[
  v(t) = \sum_{j=0}^\infty e^{-\lambda_j t} \left<v_0,\phi_j\right>\phi_j
  \]
is the solution of the abstract Cauchy problem~\eqref{parabolic}.
Using this eigenexpansion, one can deduce the assertion in a
straightforward manner from the Aronson type Gaussian bound
(\ref{eqn:Gaussian}), cf.~\cite{Bass} for a proof when $\X$ is the
reflecting Brownian motion. 
\end{proof}

\begin{lemma}
The set 
\begin{equation}\label{eqn:test}
V(D):=\{\phi : \phi\in C^2(D),\partial_{\nu}\phi=0\text{ a.e. on }\partial D \}
\end{equation}
is dense in $H^1(D)$. 
\end{lemma}
\begin{proof}
Diagonalizing the Neumann 
Laplacian on $D$ we obtain an orthonormal basis
$\{\phi_j,j\in\N\}$ of $L^2(D)$ and an increasing sequence
$(\lambda_j)_{j\in\N}$ of real positive numbers which tend to infinity
such that for every $j\in\N$, $\phi_j\in H^1(D)$ is a weak solution of
the eigenvalue problem for the Neumann Laplacian. 
Let now $\psi\in H^1(D)$ such that $\left<\psi,\phi_j\right>_{H^1(D)}=0$ 
for every $j\in\N$, then 
\begin{equation*}
\int_D\nabla\phi_j\cdot\nabla\psi\dd  x
+\int_D\phi_j\psi\dd  x=\lambda_j\int_D\phi_j\psi\dd x+\int_D\phi_j\psi\dd  x
= 0,\quad\text{for all }j\in\N.
\end{equation*} 
The fact that $\{\phi_j,j\in\N\}$ is an orthonormal basis of $L^2(D)$
implies $\psi\equiv 0$ which proves the assertion.
\end{proof}

\begin{lemma}\label{lem:smooth}
The boundary local time $L$ of $\X$
corresponding to the surface measure $\sigma$ exists as PCAF in the
strict sense.
\end{lemma}
\begin{proof}
  We have to show that the surface measure $\sigma$ of $\partial D$ has
  a bounded $1$-potential, then the assertion follows immediately from
  \cite[Theorem 5.1.6]{Fukushimaetal}.  Since we have a continuous
  transition density kernel $p$ for every $x \in \overline{D}$ and $t >
  0$, we know that the $1$-potential coincides with the $1$-resolvent
  $G_1\sigma$ of the measure $\sigma$ given by
  \[
  G_1 \sigma(x) = \int_{\partial D} \int_0^\infty
  p(t,x,y) e^{-t}\dd {t}\dd \sigma(y).
  \]
  We need to show that this is uniformly bounded. According to
  Lemma~\ref{lem:exponential}, the transition density kernel $p$ is
  uniformly bounded for every $t \ge t_0$. This together with H{\"o}lder
  continuity implies
  \[
  c_6 := \sup_{x \in \overline{D}} \int_{\partial D} \int_1^\infty
  p(t,x,y) e^{-t}\dd {t}\dd \sigma(y) < \infty.
  \]
  Therefore, it is enough to show that
  \[
  \sup_{x \in \overline{D}} \int_{\partial D} \int_0^1 p(t,x,y) \dd {t}\dd \sigma(y)
  < \infty.
  \]
  
  For every $\rho > 0$ and every $x \in D$
  such that $d(x, \partial D) \ge \rho$
  the Gaussian upper bound for the density $p$ gives
  \[
  G_1 \sigma(x) \le c_6 + c_3 \int_0^1 t^{-d/2} e^{-\rho^2 /2t} \dd {t}.
  \]
  It is straightforward to show that the integrand on the right-hand side
  has an upper bound $c_4([d \ge 3] \rho^{2-d} + [d = 2] \log \rho^{-1})$.
  Hence for every fixed $\rho > 0$ we have a uniform upper bound
  $c_7(\rho)$ for the $1$-resolvent.

  Since the boundary $\partial D$ is compact we can find $\rho > 0$ and
  a finite number of balls $B(x_k, 2\rho)$ with centers $x_k \in
  \partial D$ that cover the set $\{ x \in \overline{D} : d(x, \partial
  D) < \rho \}$. Moreover the $\rho$ can be taken so small that there exist
  bi-Lipschitz homeomorphisms between these balls and the subsets of
  $\R^d$ straightening the boundary as in the proof of
  Theorem~\ref{thm:1}.

  By the first part of the proof, we may assume $x \in B(x_k, 2\rho)$ for some $k$.
  For every $l$ such that $x \notin B(x_l,2\rho)$ the
  previous estimate holds and we have
  \[
  \sup_{x \notin B(x_l,2\rho)} \int_0^1\int_{\partial D} [y \in B(x_l,2\rho)] p(t,x,y) \dd \sigma(y)
  \dd t
  \le c_7(\rho).
  \]
  Hence it is enough to show that the contribution coming from the
  integration over those balls $B(x_l,2\rho)$ that contain $x$ is
  finite.

  When $x \in B_l := B(x_l, 2\rho)$ we use the Gaussian upper bound for a fixed
  $t \le 1$ and Lipschitz change of coordinates estimate
  $\sigma(\Psi E) \le c_D^{d-1} \sigma (E)$ of the
  $(d-1)$-dimensional surface measure $\sigma$. Therefore,
  \[
  \int_{\partial D}[y \in B_l]e^{-\lvert x-y\rvert/2t} \dd \sigma(y)
  \le c_D^{d-1} \int_{\R^{d-1}} [y \in \tilde{B}_l] e^{-\lvert x^* -
  y^*\rvert/2t} \dd y
  \]
  where $\tilde B_l := B(x^*_l, 4 c_D \rho) \subset \R^{d-1}$ is a
  $(d-1)$-dimensional ball of radius $4 c_D \rho$
  and $x^*$
  (similarly $x^*_l$ and $y^*$) is the point $x$ in the new coordinate system. 
  The integrand is
  maximized if we move $x^*$ into the center $x^*_l$ of $\tilde{B}_l$. Therefore, we can
  estimate
  \[
    \sup_{x \in B(x_l,2\rho)}\int_{\partial D} [y \in B_l] t^{-d/2} e^{-\lvert x-y\rvert^2/2t} \dd \sigma(y)
  \le c_8
  \int_0^\infty t^{-d/2}r^{d-2} e^{-r^2/2t} \dd {r}.
  \]
  Again this is straightforward to estimate and we see that integrand
  on the right-hand 
  side has an upper bound $c_5t^{-1/2}$. Since $t^{-1/2}$ is
  integrable at zero, the claim follows.
\end{proof}

\begin{lemma}\label{lem:occ}
For every $x\in\overline{D}$ and every bounded Borel measurable
function $\phi$ on $\partial D$ the following occupation formula holds:
 \begin{equation}\label{eqn:occ}
\EW_{x}\int_0^t \phi(\X_s)\dd  L_t=\int_0^t\int_{\partial
D}p(s,x,y)\phi(y)\dd \sigma(y)\dd  s\quad\text{for all }t\geq 0.
\end{equation}
\end{lemma}
\begin{proof}
By Lemma \ref{lem:smooth} the boundary local time of $\X$ exists as a
PCAF in the strict sense. 
Without loss of generality we may assume that $\phi$ is non-negative. It
follows from \cite[Theorem 5.1.3]{Fukushimaetal} that the Revuz
correspondence (\ref{eqn:Revuz}) is equivalent to 
\begin{equation*}
\EW_{\psi\cdot\dd  x}(\phi\cdot
L)_t=\int_0^t\left<\phi\cdot\sigma,T_s\psi\right>\dd  s
\end{equation*}
for every $t>0$ and all non-negative Borel measurable functions $\psi$
  and $\phi$. That is,
\begin{eqnarray*}
\int_D\psi(x)\EW_{x}\int_0^t\phi(\X_s)\dd  L_s\dd 
x&=&\int_0^t\int_{\partial D}\phi(y)T_s\psi(y)\dd \sigma(y)\dd  s\\
&=&\int_D \psi(x)\int_0^t\int_{\partial
D}\phi(y)p(s,y,x)\dd \sigma(y)\dd  s\dd x,
\end{eqnarray*} 
where we have used Fubini's Theorem. As this holds for every
non-negative Borel measurable function $\psi$, we may deduce
\begin{equation*}
\EW_{x}\int_0^t\phi(\X_s)\dd  L_s= \int_0^t\int_{\partial D}
\phi(y)p(s,x,y)\dd \sigma(y)\dd  s\quad \text{a.e. in }\overline{D}.
\end{equation*}
To obtain the assertion everywhere in $\overline{D}$ consider for
$t_0>0$ the integral
\begin{eqnarray*}
\int_Dp(t_0,x,y)\EW_{y}\int_0^{T}\phi(\X_s)\dd  L_s\dd  y,
\end{eqnarray*}
where we have set $T:=t-t_0$. Note that by the Markov property of $\X$
we may write this integral equivalently as
\begin{equation*}
\EW_{x}\int_{t_0}^t\phi(\X_s)\dd 
L_s=\int_{t_0}^t\int_{\partial D} \phi(z)p(s,x,z)\dd \sigma(z)\dd 
s\quad\text{for every }x\in\overline{D}.
\end{equation*}
Now letting $t_0\rightarrow 0$ and using the Dominated Convergence
Theorem yields the assertion.
\end{proof}

\subsection{Continuum model}
The main result for the continuum model (\ref{eqn:con}),
(\ref{eqn:continuum}) is the following Theorem.
\begin{theorem}\label{thm:cont}
Let $\kappa$ satisfy (A1) and let $f\in L^2_{\diamond}(\partial D)$ be bounded.
Then there is a unique weak solution $u\in C(\overline{D})\cap H^1_{\diamond}(D)$ to the
boundary value problem (\ref{eqn:con}), (\ref{eqn:continuum}). This solution admits the
Feynman-Kac type representation
\begin{equation}
u(x)=\lim_{t\rightarrow\infty}\EW_{x}\int_0^t f(\X_s)\dd 
L_s\quad\text{for all }x\in\overline{D}.
\end{equation}  
\end{theorem}

\begin{proof} 
The existence of a unique normalized weak solution to (\ref{eqn:con}),
(\ref{eqn:continuum}) is guaranteed by the standard theory of linear
elliptic boundary value problems. 
Let us set $u_t(x):=\EW_{x}\int_0^t f(\X_s)\dd  L_s$ and
$u_{\infty}(x):=\lim_{t\rightarrow\infty}u_t(x)$, $x\in\overline{D}$,
respectively.
From the occupation formula (\ref{eqn:occ}) and the compatibility
condition $\int_{\partial D}f(x)\dd \sigma(x)=0$ it follows immediately
that 
\begin{equation*}
u_t(x)=\int_0^t\int_{\partial D}(p(t,x,y)-\lvert
D\rvert^{-1})f(y)\dd \sigma(y)\dd  s\quad\text{for all }x\in\overline{D}.
\end{equation*}
By Lemma \ref{lem:exponential} the convergence towards the stationary
distribution is uniform over $\overline{D}$, in particular, 
\begin{equation}\label{eqn:occ1}
u_{\infty}(x)=\int_0^{\infty}\int_{\partial D}(p(t,x,y)-\lvert
D\rvert^{-1})f(y)\dd \sigma(y)\dd  s\quad\text{for all
}x\in\overline{D}. 
\end{equation}

It follows from (\ref{eqn:occ1}) together with Theorem \ref{thm:1} and
the Aronson type upper bound (\ref{eqn:Gaussian}) that $u_{\infty}$ is
in $C(\overline{D})$. Moreover Lemma \ref{lem:exponential} implies the
normalization $\int_D u_{\infty}(x)\dd  x=0$.

Now let us use the following regularization technique: Let
$(\kappa^{(n)})_{n\in\N}$ denote a sequence of smooth conductivities
with components in $C^{\infty}(\overline{D})$ such that for $1\leq
i,j\leq d$, $\kappa_{ij}^{(n)}\rightarrow\kappa_{ij}$ a.e. as
$n\rightarrow\infty$. Let us consider the Dirichlet space
$(H^1(D),\E^{(n)})$ with $\E^{(n)}(v,w):=\int_D\kappa^{(n)}\nabla
v\cdot\nabla w\dd  x$ and the associated reflecting diffusion process
$\X^{(n)}$. Using the Fukushima decomposition (\ref{eqn:Fukudecomp}) for
the coordinate functions we obtain the Skorohod decomposition
\begin{equation*}
\X_t^{(n)}=x+\int_0^ta^{(n)}(\X_s^{(n)})\dd  s+\int_0^t B^{(n)}(\X_s)\dd 
W_s-\int_0^t\nu(\X^{(n)}_s)\dd  L^{(n)}_s,
\end{equation*}  
where $W$ is a standard $d$-dimensional Brownian motion,
$a^{(n)}_i:=\sum_{j=1}^d\partial_j\kappa_{ij}^{(n)}$, $i=1,...,d,$ and
the matrix $B^{(n)}$ satisfies
$\kappa^{(n)}=\frac{1}{2}B^{(n)}(B^{(n)})^T$. Let us define $u^{(n)}_t$
in the same manner as $u_t$ and
$u^{(n)}(x):=\lim_{t\rightarrow\infty}u_{t}^{(n)}(x)$,
$x\in\overline{D}$. We show that $u^{(n)}$ is the unique weak solution
of the elliptic boundary value problem
$\nabla\cdot(\kappa^{(n)}\nabla u^{(n)})=0$ in $D$ with Neumann boundary
condition $\partial_{\nu}u^{(n)}=f$ on $\partial D$
in the Sobolev space $H^1_{\diamond}(D).$ For test functions
$v\in V(D)$ we may apply It\^{o}'s formula for semimartingales to obtain 
\begin{equation*}
\EW_xv(\X^{(n)}_t)=v(x)+\EW_x\int_0^t
\nabla\cdot(\kappa^{(n)}\nabla v(\X_s^{(n)}))\dd  s. 
\end{equation*}
By Fubini's Theorem this is equivalent to 
\begin{equation*}
T_t^{(n)}v(x)-v(x)=\int_0^t\int_Dp^{(n)}(s,x,y)\nabla\cdot(\kappa^{(n)}\nabla
v(y))\dd  y\dd  s,
\end{equation*}
where we have used the superscript \lq$(n)$\rq\ for the semigroup and
transition density kernel, respectively, corresponding to
$\kappa^{(n)}$. Multiplication with $f$, integration over $\partial D$
and another change of the orders of integration yields finally 
\begin{equation*}
\int_{\partial D}f(y)(T_t^{(n)}v(y)-v(y))\dd \sigma(y)
=\left<u_t^{(n)},\nabla\cdot(\kappa^{(n)}\nabla v)\right>.
\end{equation*}
Since $u_t^{(n)}\rightarrow u^{(n)}$ and $T^{(n)}_tv\rightarrow\lvert
D\rvert^{-1}\int_D v\dd x$, both uniformly on $\overline{D}$, as
$t\rightarrow\infty$, we have 
\begin{equation*} 
\left<u^{(n)},\nabla\cdot(\kappa^{(n)}\nabla v)\right>
=-\int_{\partial D}f(y)v(y)\dd  \sigma(y),
\end{equation*}
where we have used the expression (\ref{eqn:occ1}) with $p^{(n)}$
instead of $p$ for $u^{(n)}$. As this holds true for every $v\in V(D)$,
$u^{(n)}$ must be the unique normalized weak solution to the boundary
value problem by a density argument.

Now let us show the convergence of the sequence $(u^{(n)})_{n\in\N}$
towards $u\in H^1_{\diamond}(D)$, 
the unique solution of
(\ref{eqn:con}), (\ref{eqn:continuum}).
By the standard Trace Theorem
there exists a function $\phi\in H^1_{\text{div}}(D)$ such that
$\partial_{\nu} \phi=f$ and $\mathcal{L}\phi\in (H^1(D))'$. The bilinear
form $\E$ is coercive on $H^1_{\diamond}$, thus by the Lax-Milgram Theorem there exists a unique $w\in
H^1_{\diamond}(D)$ satisfying 
\begin{equation*}
\int_D \kappa\nabla w\cdot\nabla v\dd 
x=\left<\mathcal{L}\phi,v\right>_{(H^1_{\diamond}(D))',H^1_{\diamond}(D)}
\quad\text{for all }v\in H^1_{\diamond}(D),
\end{equation*} 
i.e., $w$ is the weak solution of the problem
$\mathcal{L}w=-\mathcal{L}\phi$ with homogeneous Neumann boundary
condition and thus $u$ has the form $u=w+\phi$. Analogously, $u^{(n)}$
has the form $u^{(n)}=w^{(n)}+\phi$. We show that
$\mathcal{L}^{(n)}\phi\rightarrow\mathcal{L}\phi$ in the norm of
$(H^1_{\diamond}(D))'$. We have for every $v\in H^1_{\diamond}(D)$
\begin{equation*}
\left<\mathcal{L}\phi-\mathcal{L}^{(n)}\phi,v\right>_{(H_{\diamond}^1(D))',H^1_{\diamond}(D)}=\sum_{i,j=1}^d\int_D(\kappa_{ij}^{(n)}-\kappa_{ij})\partial_j\phi\,\partial_i
v\dd  x.
\end{equation*}
Notice that $(\kappa_{ij}^{(n)}-\kappa_{ij})\partial_j \phi\in L^2(D)$,
$1\leq i,j\leq d$, hence the Dominated Convergence Theorem yields 
\begin{equation*}
\lvert\lvert(\kappa_{ij}^{(n)}-\kappa_{ij})\partial_j \phi\rvert\rvert_{L^2(D)}\rightarrow 0\quad\text{as }n\rightarrow\infty.
\end{equation*}
After applying H\"older's inequality we have thus shown that 
\begin{equation}\label{eqn:normconv}
\lvert\lvert \mathcal{L}\phi-\mathcal{L}^{(n)}\phi\rvert\rvert_{(H^1_{\diamond}(D))'}\rightarrow 0\quad \text{as }n\rightarrow\infty.
\end{equation}

Moreover from our assumptions on the sequence $(\kappa^{(n)})_{n\in\N}$ it is clear that for $1\leq i,j\leq d$, the functions $\lvert\kappa_{ij}^{(n)}-\kappa_{ij}\rvert^2$ are measurable and bounded and $\lvert\kappa_{ij}^{(n)}-\kappa_{ij}\rvert^2\rightarrow 0$ for a.e. $x\in\overline{D}$ as $n\rightarrow\infty$. Hence the Dominated Convergence Theorem yields $\lvert\lvert \kappa_{ij}^{(n)}-\kappa_{ij}\rvert\rvert_{L^2(D)}\rightarrow 0$ as $n\rightarrow 0$. It is well-known, that this implies $G$-convergence of the sequence of elliptic operators $(\mathcal{L}^{(n)})_{n\in\N}$ towards $\mathcal{L}$, cf.~\cite{Zikov}. By \cite[Theorem 5]{Zikov} this $G$-convergence together with the convergence (\ref{eqn:normconv}) yields that $w^{(n)}\rightarrow w$ weakly in $H^1_{\diamond}(D)$, thus implying $u^{(n)}\rightarrow u$ weakly in $H^1_{\diamond}(D)$.

On the other hand by \cite[Lemma 2.2]{RoSl00} together with the H\"older continuity up to the boundary of both, $p^{(n)}$, $n\in\N$, and $p$, it follows that for fixed $x\in\overline{D}$,  $p^{(n)}(\cdot,x,\cdot)\rightarrow p(\cdot,x,\cdot)$ uniformly on compacts in $(0,T]\times\overline{D}$ for all $T>0$. It follows from (\ref{eqn:occ1}) that $u^{(n)}(x)\rightarrow u_{\infty}(x)$ for all $x\in\overline{D}$ as $n\rightarrow\infty$. In particular $u$ must coincide with $u_{\infty}$ and the assertion is proved.

\end{proof}
\begin{remark}\label{rem:1}
Note that a similar regularization technique may be used to prove the Feynman-Kac formula 
\begin{equation*}
u(x)=\EW_x\phi(\X_{\tau_D}),\quad x\in\overline{D}
\end{equation*}
for the conductivity equation (\ref{eqn:con}) with Dirichlet boundary condition $u|_{\partial D}=\phi$, where $\tau_D$ denotes the first exit time from the domain $D$. Such a proof requires the fact that for every $x\in \overline{D}$,
\begin{equation*}
\mathbb{P}_x\circ (\X^{(n)},\tau_D^{(n)})^{-1}\rightarrow \mathbb{P}_x\circ (\X,\tau_D)^{-1}\quad\text{as }n\rightarrow\infty
\end{equation*}
in the topology of $C([0,\infty);\R^d)\times\R$ which follows easily from the assumption that $D$ has a Lipschitz boundary, i.e., all points of $\partial D$  are regular, cf.~\cite[Section 4.27]{Karatzas}.
\end{remark}

\subsection{Complete electrode model}
The main result for the complete electrode model (\ref{eqn:con}),
(\ref{eqn:cem}) is the following Theorem.
\begin{theorem}\label{thm:cem}
For given functions $f,g$ defined by (\ref{eqn:functions}) and a voltage
pattern $U\in\R^N$ satisfying (\ref{ground}), there is a unique
weak solution $u\in C(\overline{D})\cap H^1(D)$ to the boundary value problem
(\ref{eqn:con}), (\ref{eqn:cem}). This solution admits the Feynman-Kac
type representation 
\begin{equation}\label{eqn:FK}
u(x)=\EW_{x}\int_0^{\infty}e_g(t)f(\X_t)\dd 
L_t\quad\text{for all }x\in\overline{D},
\end{equation}
with 
\begin{equation}
e_g(t):=\exp\Big(-\int_0^tg(\X_s)\dd  L_s\Big),\quad t\geq 0.
\end{equation}
\end{theorem}
Before we are ready to give a proof of Theorem \ref{thm:cem} let us
introduce the \emph{Feynman-Kac semigroup} of the complete electrode
model, i.e., the one-parameter family of operators $\{T^g_t,t\geq 0\}$
defined by
\begin{equation} 
T^g_tv(x):=\EW_x e_g(t)v(\X_t),\quad x\in\overline{D}\text{ and } t\geq 0.
\end{equation}
The following Theorem is crucial for proving the claimed regularity of the potential.
\begin{theorem}\label{thm:3}
$\{T^g_t,t\geq 0\}$ is a strong Feller semigroup on $L^2(D)$.
\end{theorem}
\begin{proof}
To show that $\{T^g_t,t\geq 0\}$ is a strongly continuous semigroup on
$L^2(D)$, one can employ the theory of symmetric Dirichlet spaces. To be
precise, one must show that $\{T^g_t,t\geq 0\}$ is associated with the
perturbed Dirichlet space $(\D(\E^g),\E^g)$, which is obtained by
perturbation of $(\D(\E),\E)$ with the measure $-g\cdot\sigma$, i.e.,
\begin{equation*}
\E^g(v,w)=\E(v,w)+\int_{\partial D}g(x)v(x)w(x)\dd \sigma(x),\quad
v,w\in\D(\E^g)=H^1(D),
\end{equation*}
where the identity $\D(\E^g)=H^1(D)$ follows from the standard
Trace Theorem.  As in the proof of \cite[Theorem 6.1.1]{Fukushimaetal},
it is sufficient to show $G_{\alpha}^g\phi\in H^1(D)$,
$\E_{\alpha}^{g}(G_{\alpha}^g\phi,v)=\left<\phi,v\right>$ for all
$\phi\in L^2(D)$ and $v\in H^1(D)$, where $G_{\alpha}^g\phi$ denotes the
Laplace transform
\begin{equation*}
G_{\alpha}^g\phi(x)=\EW_x\int_0^{\infty}e_g(t)e^{-\alpha
t}\phi(\X_t)\dd  t.
\end{equation*} 
We omit this computation for brevity. Moreover $T^g_t$ is a bounded
operator from $L^1(D)$ to $L^{\infty}(D)$ for every $t>0$ which can be
shown using Fatou's Lemma. By the Dunford-Pettis Theorem $T^g$ can thus
be represented as an integral operator for every $t>0$,
\begin{equation}
T^g_t\phi(x)=\int_Dp^g(t,x,y)\phi(y)\dd  y\quad\text{a.e. on }D,\text{
for every }\phi\in L^1(D), 
\end{equation}
where for all $t>0$ we have $p^g(t,x,y)\in L^{\infty}(D\times D)$ and
$p^g(t,x,y)\geq 0$ for a.e. $x, y \in \overline{D}$.  For the strong Feller
property we have to show
that $T^g_t$, $t>0$ maps bounded measurable functions to
$C(\overline{D})$. We use the method from the papers
\cite{Hsu,Papanicolaou} to construct the transition kernel density
$p_g$. Let $p^g_0(t,x,y):=p(t,x,y)$ and set 
\begin{equation*}
p_k^g(t,x,y):=\int_0^t\int_{\partial
D}p(s,x,z)g(z)p_{k-1}^g(t-s,z,y)\dd \sigma(z)\dd  s,\quad k\in\N.
\end{equation*}
Note that the terms $p_k^g$ are positive and symmetric in the $x$ and
$y$ variables by the properties of $p$. By induction using Lemma
\ref{lem:occ} it is not difficult to verify that for all $k\in\N$
\begin{equation*}
\int_0^t\int_{\partial D}g(x)p_k^g(s,x,y)\dd \sigma(x)\dd  s\leq
\Big(\sup_{x\in\overline{D}}\Big\{\EW_x\int_0^tg(\X_s)\dd 
L_s\Big\}\Big)^{k+1}
\end{equation*}
and that there is a positive constant $c_1$ such that 
\begin{equation}\label{eqn:densest}
p_k^g(t,x,y)\leq c_1^{k+1}t^{-d/2}\Big(\sup_{x\in\overline{D}}
\Big\{\EW_x\int_0^tg(\X_s)\dd  L_s\Big\}\Big)^{k+1}\quad
\text{for all }k\in\N.
\end{equation}
Let us show the continuity of $p^g_k$, $k\in\N\cup\{0\}$. For $k=0$ this
is a consequence of Theorem \ref{thm:1}. Now assume that $p^g_{k-1}$ is
continuous on $(t_0,T]\times\overline{D}\times\overline{D}$ for $t_0>0$,
then we have for $t\in (t_0,T]$
\begin{equation*}\begin{split}
p_k^g(t,x,y)=&\int_0^{t_0}\int_{\partial D}p(s,x,z)g(z)
p_{k-1}^g(t-s,z,y)\dd \sigma(z)\dd  s\\
&+\int_{t_0}^{t}\int_{\partial
D}p(s,x,z)g(z)p_{k-1}^g(t-s,z,y)\dd \sigma(z)\dd  s.
\end{split}
\end{equation*}
Note that the first integral on the right-hand side tends to zero
uniformly as $t_0\rightarrow 0$, which is a consequence of
(\ref{eqn:densest}), while the second integral is continuous by
assumption. Hence there exists a $T>0$ such that the series
$p^g(t,x,y):=\sum_{k=0}^{\infty}p_k^g(t,x,y)$ converges absolutely and
uniformly on any compact subset of
$(0,T]\times\overline{D}\times\overline{D}$ and is thus continuous on
$(0,T]\times\overline{D}\times\overline{D}$. By the Markov property we
have for all $t\in (0,T]$ and every $x\in\overline{D}$ the following 
expression for $T^g_t\phi(x)$:
\begin{equation*}
\int_Dp^g(t,x,y)\phi(y)\dd  y=\EW_x\phi(\X_t)+\sum_{k=1}^{\infty}
\frac{1}{k!}\EW_x\Big\{\Big(\int_0^tg(\X_s)\dd 
L_s\Big)^k\phi(\X_t)\Big\}.
\end{equation*}
The assertion for arbitrary $T>0$ follows from the Chapman-Kolmogorov
equation.
\end{proof}
 
\begin{proof}[Proof of Theorem \ref{thm:cem}]
First we show that $u\in C(\overline{D})$. Let us define a martingale
with respect to $\{\F_t,t\geq 0\}$ by 
\begin{equation*}
\EW_x\Big\{\int_0^{\infty}e_g(s)f(\X_s)\dd  L_s|\F_t\Big\}
=\int_0^te_g(s)f(\X_s)\dd  L_s+e_g(t)u(\X_t),
\end{equation*}
where the right-hand side is obtained using the Markov property of $\X$
together with the fact that $e_g$ is a multiplicative functional of
$\X$. Obviously
\begin{equation*}
e_g(t)u(\X_t)-u(x)+\int_0^te_g(s)f(\X_s)\dd  L_s
\end{equation*} 
is a martingale with respect to $\{\F_t,t\geq 0\}$ as well and hence we
have for all $0\leq s\leq t$:
\begin{equation*}
e_g(s)u(\X_s)=e_g(s)\EW_{\X_s}e_g(t-s)u(\X_{t-s})+e_g(s)
\EW_{\X_s}\int_0^{t-s}e_g(r)f(\X_r)\dd  L_r.
\end{equation*}
Setting $s=0$ yields thus 
\begin{equation*}
u(x)=T^g_t u(x)+\EW_x\int_0^t e_g(s)f(\X_s)\dd  L_s
\quad\text{for all }t\geq 0. 
\end{equation*}
By the Markov property we have $T_t^gu(x)=T^g_s(T^g_{t-s}u)(x)$ and
$T^g_tu$ is continuous on $\overline{D}$ by Theorem \ref{thm:3}. To
prove that $u$ is continuous on $\overline{D}$ it is sufficient to show
that the second term on the right-hand side tends to zero uniformly in
$x$ as $t\rightarrow 0$. This is, however, clear since we may estimate
\begin{equation*}
\sup_{x\in\overline{D}}\Big\{\EW_x\int_0^te_g(s)f(\X_s) \dd 
L_s\Big\}\leq
z^{-1}\max_{l=1,...,N}\{U_l\}\sup_{x\in\overline{D}}\{\EW_xL_t\},
\end{equation*}
where the right-hand side tends to zero as $t\rightarrow 0$ by Lemma
\ref{lem:occ}.

It remains to show that $u$ is given by (\ref{eqn:FK}). Note
first that the \emph{gauge function} $\EW_x\int_0^{\infty}e_g(t)\dd  L_t$ is
$\mathbb{P}_x$-a.s. bounded for every $x\in\overline{D}$, hence the
expression (\ref{eqn:FK}) is well-defined.
By the Lax-Milgram Theorem there exists a weak solution of the boundary
value problem (\ref{eqn:con}), (\ref{eqn:cem}) such that for every $v\in
H^1(D)$
\begin{equation*}
\E(u,v)=\int_{\partial D}f(x)v(x)\dd \sigma(x)- \int_{\partial
D}g(x)u(x)v(x)\dd \sigma(x).
\end{equation*}
By standard theory of linear elliptic boundary value problems $u$ is
bounded, cf.~\cite{Ladyzhenskaya}, which implies by \cite[Theorem
5.4.2]{Fukushimaetal} together with the Fukushima decomposition
(\ref{eqn:Fukudecomp}) that for q.e. $x\in\overline{D}$, $\mathbb{P}_x\text{-a.s.}$
\begin{equation*}
\tilde{u}(\X_t)=\tilde{u}(x)+\int_0^t\nabla \tilde{u}(\X_s)\dd 
M_s-\int_0^t f(\X_s)\dd  L_s+\int_0^t g(\X_s) \tilde{u}(\X_s)\dd 
L_s.
\end{equation*}
Note that the second term on the right-hand side is a local martingale
with respect to $\{\F_t,t\geq 0\}$ and that $e_g$ is continuous, adapted
to $\{\F_t,t\geq 0\}$ and of bounded variation. Multiplication by such functions leaves
the class of local martingales invariant.
Using integration by parts we obtain thus for q.e.
$x\in\overline{D}$ and $t\geq 0$ the identity
\begin{equation*}
\tilde{u}(\X_t)e_g(t)=\tilde{u}(x)+\int_0^te_g(s)\nabla
\tilde{u}(\X_s)\dd  M_s-\int_0^t e_g(s)f(\X_s)\dd  L_s,
\end{equation*}
where the second summand on the right-hand side is a local martingale.
That is, there exists an increasing sequence $(\tau_k)_{k\in\N}$ of
stopping times which tend to infinity such that for every $k\in\N$
$$\Big\{\int_0^{t\wedge\tau_k}e_g(s)\nabla \tilde{u}(\X_s)\dd  M_s,t\geq
0\Big\}$$
is a martingale with respect to $\{\F_t,t\geq 0\}$. In particular we
have for q.e. $x\in\overline{D}$ and every $k\in\N$
\begin{equation*}
\tilde{u}(x)=\EW_x\int_0^{t\wedge\tau_k}e_g(s)f(\X_s)\dd 
L_s+\EW_x\tilde{u}(\X_{t\wedge\tau_k})e_g(t\wedge\tau_k)\quad\text{for
all }t\geq 0.
\end{equation*}
By the uniform integrability of $\{e_g(s),0\leq s\leq t\}$ with respect
to $\mathbb{P}_x$, $x\in\overline{D}$, $t>0$, together with the Monotone
Convergence Theorem we obtain 
\begin{equation*}
\tilde{u}(x)=\EW_x\int_0^te_g(s)f(\X_s)\dd  L_s+
\EW_x\tilde{u}(\X_t)e_g(t)\quad\text{for q.e.  }x\in\overline{D}.
\end{equation*}
Letting $t\rightarrow\infty$ finally yields 
\begin{equation*}
\tilde{u}(x)=\EW_x\int_0^{\infty}e_g(t)f(\X_t)\dd 
L_t\quad\text{for q.e. }x\in\overline{D},
\end{equation*}
where we have used the fact that $\tilde{u}$ is bounded. As we have
shown that the right-hand side in the last equality is continuous up to
the boundary, the function $u$ coincides with its quasi-continuous
version $\tilde{u}$ and the assertion holds for every
$x\in\overline{D}$.
\end{proof}
\begin{remark}
Note that the technique we used to prove Theorem \ref{thm:cem} fails for
the Neumann problem corresponding to the continuum model. This comes
from the fact that in this case the gauge function becomes infinite.
For the same reason Theorem 1.2 from \cite{ChenZhang}, specialized to a
zero lower-order term, does not yield the desired Feynman-Kac type
formula for the continuum model either.
\end{remark}
We can now generalize the martingale characterization obtained in
\cite{Papanicolaou} for weak solutions of (\ref{eqn:con}),
(\ref{eqn:cem}). 
\begin{theorem}
Suppose the conditions of Theorem \ref{thm:cem} are satisfied. Then the following statements are equivalent:
\begin{itemize}
\item[(i)] $u$ is the weak solution of the boundary value problem
(\ref{eqn:con}), (\ref{eqn:cem})
\item[(ii)] For every $x\in\overline{D}$ the expression 
\begin{equation}
\mathcal{M}_t(u):=u(\X_t)-u(x)-\int_0^tg(\X_s)u(\X_s)\dd 
L_s+\int_0^tf(\X_t)\dd  L_s
\end{equation}
is a continuous martingale with respect to
$\{\F_t,t\geq 0\}$. 
\end{itemize}
\end{theorem}
\begin{proof}
Let us assume that (i) holds. First recall that the integral with
respect to $L$ is defined in pathwise Lebesgue-Stieltjes sense with
respect to the induced random measure $\lambda((s,t]):=L_t-L_s$ on
$\R_+\cup\{0\}$, which is absolutely continuous in the sense that for any
$\varepsilon>0$ there exists a $\delta>0$ such that $\int_B\lvert
\phi\rvert\dd \lambda\leq\varepsilon$ holds for all measurable sets $B$
with $\lambda(B)<\delta$ and bounded measurable functions $\phi$. In
particular this implies the continuity of the maps $t\mapsto\int_0^tf\dd 
L_s$ and $t\mapsto\int_0^tgu\dd  L_s$, respectively. The solution $u$ is continuous up
to the boundary by Theorem \ref{thm:cem}. By the Markov property of $\X$
we have for all $s\leq t$:
 \begin{eqnarray*}
 \EW_x\{\mathcal{M}_t(u)|\F_s\}&=&\EW_{\X_s}\Big\{u(\X_{t-s})+\int_0^{t-s}(f(\X_r)-g(\X_r)u(\X_r))\dd 
 L_r\Big\}\\
 &&+\int_0^s(f(\X_r)-g(\X_r)u(\X_r))\dd  L_r-u(x)\\
 &=&\EW_{\X_s}\mathcal{M}_{t-s}(u)+\mathcal{M}_s(u)
 \end{eqnarray*}
Thus in order to obtain (ii) it suffices to show that
$\EW_x\mathcal{M}_t(u)=0$ $\mathbb{P}_x$-a.s. for all $t\geq 0$ and
all $x\in\overline{D}$. 
From standard theory of strongly continuous semigroups it is known that
$\EW_xu(\X_t)=T_t u(x)$, considered as a Banach space valued mapping
from $\R_{+}$ to $H^1(D)$, is continuously differentiable for every
$t>0$, cf.~\cite{Pazy}. Note that we may write $p(s,x,y)=T_{s-t}p(t,y,x)$
for every $s\geq t>0$ so that $p$, considered as a Banach space valued
  mapping is also continuously differentiable with derivative
  $\frac{\mathrm{d}}{\mathrm{d} s} p(s,x,y)=-\mathcal{L}_yp(s,x,y)$.
We may thus differentiate the expression $\EW_xu(\X_s)$ under the
integral sign to obtain
\begin{equation*}
\begin{split}
\frac{\mathrm{d}}{\mathrm{d} s}\EW_xu(\X_s)=&\int_D
u(y)\frac{\mathrm{d}}{\mathrm{d} s} p(s,x,y)\dd  y=-\int_{\partial
D}p(s,x,y)f(y)\dd \sigma(y)\\
&+\int_{\partial D}p(s,x,y)g(y)u(y)\dd \sigma(y),
\end{split}
\end{equation*}
where we have used (i) together with the fact that for fixed $s>0$ and
$x\in\overline{D}$ the function $p(s,x,\cdot)$ belongs to $H^1(D)$. 
By integration form $0$ to $t$ together with Lemma \ref{lem:occ} we
arrive at
\begin{equation*}
\EW_xu(\X_t)-u(x)=\EW_x\Big\{\int_0^t(-f(\X_s)+g(\X_s)u(\X_s))\dd  L_s\Big\}
\end{equation*}
and (ii) is proved. 

Now let us assume that (ii) holds. By the continuity of $u$
and uniqueness of the Doob-Meyer decomposition the term $u(\X_t)$ is a
continuous semimartingale with respect to $\{\F_t,t\geq 0\}$. Moreover
$e_g(t)$ is continuous, adapted to $\{\F_t,t\geq 0\}$ and of bounded
variation. Multiplication by such functions leaves the class of
semimartingales invariant, i.e., $e_g(t)u(\X_t)$, $t\geq 0$ is a
continuous semimartingale as well. In particular we may define the
It\^{o} stochastic integral with respect to this semimartingale and
integration from $0$ to $t$ of the expression $\dd 
(e_g(s)u(\X_s))+\exp(-\int_0^sg(\X_r)f(\X_s)\dd  L_s)$ yields another
martingale with respect to $\{\F_t,t\geq 0\}$, namely
\begin{equation*}
e_g(t)u(\X_t)-u(x)+\int_0^te_g(s)f(\X_s)\dd  L_s.
\end{equation*}
Let $v$ denote the unique solution to (\ref{eqn:con}), (\ref{eqn:cem}),
then we know from the proof of Theorem \ref{thm:cem} that 
\begin{equation*}
e_g(t)v(\X_t)-v(x)+\int_0^te_g(s)f(\X_s)\dd  L_s
\end{equation*}
is a martingale with respect to $\{\F_t,t\geq 0\}$. Hence
$e_g(t)(u(\X_t)-v(\X_t))$ is a martingale with respect to $\{\F_t,t\geq
0\}$ and by taking the expectation we obtain $u(x)-v(x)=T_t^g(u-v)(x)$.
That is,  $\E^g(u-v,w)=0$ for every $w\in H^1(D)$ or equivalently
\begin{equation*}
\E^g(u,w)=\E^g(v,w)=\int_{\partial D}f(x)w(x)\dd \sigma(x)\quad \text{for
every }w\in H^1(D)
\end{equation*}
which is statement (i).
\end{proof}
\section{Probabilistic interpretation of the inverse conductivity problem}
\label{section5}
The inverse conductivity problem for the continuum model, the so-called 
\emph{Calder\'on problem} reads as follows: \emph{Given the Cauchy data on the 
boundary, i.e., all pairs of voltage and current patterns
 $(\phi,\partial_{\kappa\nu}u)$, each pair corresponding to a 
 solution of the conductivity equation (\ref{eqn:con}) with 
 $u|_{\partial D}=\phi$, is it possible to determine the conductivity $\kappa$ uniquely? }

Since we assume that $D$ is a Lipschitz domain
and the conductivity is uniformly elliptic and bounded,
the solution of the Dirichlet boundary value problem is unique.
Therefore, the Cauchy data can be described as a Dirichlet-to-Neumann
map
\[
\Lambda_\kappa \colon \phi \mapsto \psi = \partial_{\kappa\nu} u, \quad
H^{1/2}(\partial D) \to H^{-1/2}(\partial D)
\]
where both the domain and the range are given by the standard Trace Theorem.
This means that the Calder\'on problem can be restated as \emph{given
$\Lambda_\kappa$, is it possible to determine $\kappa$ uniquely?}. 

We have already demonstrated that solving the forward problem for
the conductivity equation is intimately connected with the diffusion
process $\X$. 
We start with the reflecting diffusion $\X$ and we stop it at the first exit time $\tau_D$
from the domain $D$, leading to the representation of the solution as
\[
u(x) = \EW_x \phi(\X_{\tau_D}),\quad x\in\overline{D},
\]
cf.~Remark \ref{rem:1}.
Therefore, the forward problem related to the conductivity equation
could be probabilistically interpreted as \emph{given $\X$ and the boundary
data $\phi$ determine the corresponding potential $u$}.

Another way of thinking of the forward problem would be to just
recover $\Lambda_\kappa$ given $\kappa$ since this is the actual inverse
of the inverse problem. Since $\mathcal{L}=\nabla\cdot\kappa\nabla$ is the infinitesimal
generator of the Markov process $\X$, we are tempted to seek for a
Markov process $\vX$ with the generator $\Lambda_\kappa$. This
observation was first made by Hsu in 1986 for the reflecting Brownian
motion~\cite{Hsu2}. The Dirichlet-to-Neumann map generates the so-called
\emph{boundary process} $\vX$ associated with the Markov process $\X$, which 
we shall define below.
This way the probabilistic interpretation of the forward problem could be stated as 
\emph{given $\X$ determine the associated boundary process
$\vX$}.

This leads to the following probabilistic interpretation of the Calder\'on problem:
\emph{Given a boundary process $\vX$, is it possible to uniquely determine a process $\X$ whose
boundary process $\vX$ is?}.

Let us now show that this interpretation can be carried out rigorously in our
setting. The boundary local time $L$ is a nondecreasing, adaptive process
that increases only when $\X$ is on the boundary. Following~\cite{Fukushimaetal}, 
we define the right-continuous
right-inverse $\tau$ of $L$ by 
\begin{equation}
  \label{RCRI}
  \tau(s) := \sup\{ r \ge 0 : L_r \le s \}.
\end{equation}
The random variable $\tau(s)$, $s \in [0,\infty)$, is a stopping time with
respect to the right-continuous history $(\mathcal F_t)$ of $\X$ since
$\{ \tau(s) \ge t \} = \{ L_t \le s \} \in \mathcal F_t$ and moreover,
by continuity of the sample paths of $\X$ we see that for every $s \in [0,\infty)$, the
process $\X$ is on the boundary $\partial D$ at time $\tau(s)$.
\begin{definition}
  We define the \emph{boundary process} $\vX$ associated with $\X$ as
  the time-changed trace
  \[
  \vX_t := \X_{\tau(t)}
  \]
  and the \emph{boundary filtration}
  \[
  \widehat{\mathcal F}_t := \mathcal F_{\tau(t)}.
  \]
\end{definition}
We know that the boundary local time $L$ is a PCAF in the strict sense by
Lemma~\ref{lem:smooth} and therefore, the boundary $\partial D$ is the
quasi support of $L$, cf.~\cite[Theorem 5.1.5]{Fukushimaetal}. Moreover,
since we have a Lipschitz domain, every boundary point is a
regular point, and therefore, the
boundary process is a Hunt process on the boundary $\partial D$, cf.~\cite[Theorem
A.2.12., Theorem 6.2.1]{Fukushimaetal}.

In the sequel, we will denote the transition
semigroup of the boundary process by $\{\vT_t, t \ge 0\}$ and the generator of
the semigroup by $\widehat{\mathcal{L}}$.

We note that the representation Theorem~\ref{thm:cont} can be expressed
with the help of the boundary process $\vX$, the first exit time
$\tau_D$ and the first exit place $X_{\tau_D}$.
\begin{lemma}
  \label{lemma:cont2}
    Suppose the conditions of Theorem~\ref{thm:cont} are satisfied.
    Then the solution has a representation
    \[
    u(x) = \lim_{t\rightarrow \infty} \EW_{x} v(X_{\tau_D},t)
    \]
    where 
    \[
    v(x,t) := \int_0^t \EW_{x} f(\vX_s) \dd  s\quad\text{for all }x\in\overline{D}.
    \]
\end{lemma}
\begin{proof}
This follows from Theorem~\ref{thm:cont} by the
strong Markov property and the change of variables formula
\begin{equation}\label{CVF}
 \int_{\tau(a)}^{\tau(b)} f(s) \dd  L_s = \int_a^b f(\tau(s)) \dd  s 
\end{equation}
which follows by Monotone Class Theorem from the observation that
$$[L_a,L_b] = \tau \circ g_{a,b},$$ where we have set $g_{a,b}(t):= [ t \in [a,b] ]$.
\end{proof}

We will next verify the observation of Hsu~\cite{Hsu2} for our setting.
\begin{theorem}\label{bdry:generator}
  Suppose assumption (A2) holds.
  The infinitesimal generator of the boundary process $\vX$
  coincides with the Dirichlet-to-Neumann map $\Lambda_\kappa$ on 
  $H^{3/2}(\partial D)$.
\end{theorem}
\begin{proof}
When $\phi \in H^{3/2}(\partial D)$, the conductivity equation (\ref{eqn:con}) 
with Dirichlet boundary value $\phi$ admits a
solution in $H^2(D)$. We may apply the Fukushima decomposition
\[
\tilde{u}(\X_t)=\tilde{u}(\X_0)+M_t^{[u]}+N_t^{[u]}
\]
for all $t> 0$ and all $x \in \partial D$ since the Revuz measures
corresponding to $M^{[u]}$ and $N^{[u]}$ are smooth measures of $\X$ when $u \in H^2(D)$ and
$\partial_{\kappa\nu} u \in H^{1/2}(\partial D)$. We may hence assume that $u$
itself is quasi-continuous, i.e., $u = \tilde{u}$. Since $u$ solves the
conductivity equation, we have
\[
N_t^{[u]} = \int_0^t \nabla \cdot (\kappa\nabla u) (\X_s) \dd  s -
\int_0^t \partial_{\kappa \nu} u (\X_s) \dd  L_s
= -\int_0^t \partial_{\kappa \nu} u (\X_s) \dd  L_s.
\]
Since $M^{[u]}$ is a martingale AF, we obtain that the process
\[
Z_t = u(\X_t)-u(\X_0) +\int_0^t \partial_{\kappa \nu} u (\X_s) \dd  L_s
\]
is a martingale. Since $\tau(t)$ is a stopping time, we may apply the
Optional Stopping Theorem and we obtain
\[
0 = \EW_x Z_{\tau(t)} = \EW_x \phi(\vX_t)-\phi(x)
+\EW_x \int_0^{\tau(t)} \partial_{\kappa \nu} u (\X_s) \dd  L_s
\]
for every $x \in \partial D$.  After an application of the change of variables
formula~\eqref{CVF} we arrive to
\begin{equation}
  \label{eq:IG1}
  \EW_x \phi(\vX_t) = \phi(x)
  -\EW_x \int_0^t \partial_{\kappa \nu} u (\vX_s) \dd  s=
  \phi(x)-\EW_x \int_0^t \Lambda_\kappa \phi (\vX_s) \dd  s.
\end{equation}
The identity~\eqref{eq:IG1} applied to $\widehat T_r \phi$ gives
\[
\vT_{t+r} \phi = \vT_t(\vT_r \phi) = \vT_r \phi - \int_0^t
\vT_s\Lambda_\kappa \vT_r \phi \dd  s.
\]
Therefore, the function $v(r) = \vT_r \phi$ solves the abstract Cauchy
problem~\eqref{parabolic} with both $\Lambda_\kappa$ and the generator
$\widehat{\mathcal{L}}$ of the boundary semigroup in place of
$\mathcal{L}$ in~\eqref{parabolic} for the intial value $\phi$. This proves
the claim.
\end{proof}
\begin{remark}
  In this section we will not try to obtain the optimal regularity
  assumptions for the conductivity $\kappa$. Neither will we try to
  analyze the optimal regularity needed for the domain $D$. For instance, the
  assumption (A2) in
  Theorem~\ref{bdry:generator} is clearly not optimal. The assumption is
  needed for the elliptic regularity so that we have a simple
  representation for the additive functional $N^{[u]}$ with zero energy. This
  could be improved by using an approximation procedure similar to the one
  we used in the proof of Theorem~\ref{thm:cont}.
  However, we will leave these improvements for future work.

\end{remark}

We can now elaborate the probabilistic inverse problem a bit further.
For every given $\kappa$ we have an associated diffusion process
$\X_\kappa$. The above construction associates the diffusion process
$\X_\kappa$ with its boundary process $\vX_\kappa$. Suppose we are given
a boundary process $\vX$ and we know a priori that there exists at least
one $\kappa_0$ such that $\vX = \vX_{\kappa_0}$. Note that equality in
the sequel means equality in distribution. The uniqueness question
related to the Calder\'on problem would then be recasted as \emph{suppose
there exists a $\kappa$ such that $\vX = \vX_{\kappa}$. Does it
follow that $\X_\kappa = \X_{\kappa_0}$?} The reconstruction problem
can be stated as \emph{reconstruct the process $\X$ such that
$\vX_{\kappa} = \vX$}.

The Calder\'on problem in $2$ dimensions is known to be solvable for 
isotropic $\kappa \in L^{\infty}(D)$. Given the boundary
process $\vX = \vX_{\kappa_0}$ we can thus uniquely determine the generator
$\Lambda = \Lambda_{\kappa_0}$. The celebrated result of
Astala and P\"aiv\"arinta~\cite{AstalaP} says that whenever $\Lambda_\kappa =
\Lambda$ and both $\kappa$ and $\kappa_0$ are isotropic, 
uniformly bounded and uniformly elliptic, then $\kappa = \kappa_0$.
Therefore, the equality $\X_\kappa = \X_{\kappa_0}$ must hold as well.

The recent result by Haberman and Tataru~\cite{HabermanTataru} implies
the same for higher dimensional cases when $\kappa$ and $\kappa_0$ are
assumed to be $C^1(D)$ or if they are Lipschitz continuous and close to identity
in certain sense.

When the conductivity is not assumed to be isotropic the uniqueness has
always an obstruction, namely we have $\Lambda_{\kappa_0}=
\Lambda_{\kappa_1}$, whenever $\kappa_1 = F_* \kappa_0$ is the
push-forward conductivity by a diffeomorphism $F$ on $D$ that leaves the
boundary $\partial D$ invariant. In the plane, this is known to be the
only obstruction by the result of Astala, Lassas and
P\"aiv\"arinta~\cite{AstalaLP} which holds without additional regularity assumptions
on the conductivity. In higher dimensions the question is still very
much open in general, see~\cite{DSFKSU} for further discussion.

These results from analysis all rely on so-called \emph{complex geometric
optics solutions} and the authors are not aware of any probabilistic interpretation
of these solutions. Therefore, a probabilistic solution
to the (probabilistic interpretation of the) Calder\'on problem should
use some other techniques.

One possible approach could be to understand the structure of the 
boundary process more thoroughly and attempt to ``join the dots'' 
by constructing the compatible excursions between the boundary 
points or by showing the uniqueness of the distribution of the 
compatible excursions.
As a first step towards that direction we adapt the representation 
result from~\cite{Hsu2} to our setting. The following Theorem is the main result of this section.
\begin{theorem}
  \label{lemma:integrodiff}
  Suppose both assumptions $(A1)$ and $(A2)$ hold. Then the Dirichlet-to-Neumann map $\Lambda_\kappa$ is of form
  \[
  \Lambda_\kappa \phi = b \cdot \nabla_T \phi + A_0
  \]
  where $b$ is a vector field given by
  \[
  b := \Lambda_\kappa \mathrm{id}.
  \]
  The operator $A_0$ is the integral operator
  \[
  A_0 \phi(x): = 2\int_{\partial D} A_1(x) \phi(y) N(x,y) \dd \sigma(y),
  \]
  where $A_1(x) \phi(y) := \phi(y)-\phi(x) - \nabla_T \phi(x) \cdot
  (y-x)$ and $N$ is
  the conormal derivative of the Poisson kernel, explicitly given by the
  transition density kernel $p_0$ of the killed diffusion $\X_0$ as
  \[
  N(x,y) = \int_0^\infty \partial_{\kappa\nu(x)} \partial_{\kappa\nu(y)}
  p_0(t,x,y) \dd  t.
  \]
\end{theorem}
From this representation, we see that the generator of the boundary
process $\vX$ is of form
\[
\widehat{\mathcal{L}} = \nabla \cdot \widehat A \nabla + b \cdot
\nabla + A_0,
\]
with diffusion coefficient $\widehat A = 0$. This is an
integro-differential operator in the sense of Lepeltier and
Marchal~\cite{LepeltierM}. As in Hsu~\cite{Hsu2}, this means that $\vX$
is a pure jump process without diffusion part and 
the jump distribution can be described with the help
of~\cite[Th{\'e}or{\`e}me 10]{LepeltierM}.
\begin{lemma}
  \label{lemma:jumps}
   For every Borel measurable $\phi \colon \partial D \times \partial D \to
   \R_+$ vanishing on the diagonal and any stopping time $\tau$ for $\vX$, we have
   \[
   \EW_x \sum_{s \le \tau} \phi(\vX_{s-}, \vX_s)[\vX_{s-} \ne \vX_s] = 2\EW_x \int_0^\tau
   \int_{\partial D} \phi(\vX_s, y) N(\vX_s, y)\dd \sigma(y)\dd  s.
   \]
\end{lemma}
\begin{proof}
  Suppose first that $\textrm{diam}(D) \le 1$. We note
  that the operator $A_0$ in Lemma~\ref{lemma:integrodiff} coincides with
  the integral operator causing the jumps, namely when $\phi \in
  H^{3/2}(\partial D)$
  and it is continued as $\psi \in H^2(\R^d)$ so that $\phi$ and its
  tangential derivative are continued as constants along the conormal
  directions in the neighborhood of the boundary $\partial D$, then for
  every $x \in \partial D$ we have
  \[
  A_0\phi (x) = \int_{\R^d \setminus \{0\}} \big(\psi(x+z) - \psi(x) - \big[\lvert z\rvert\le
  1\big] z \cdot \nabla \psi(x) \big)
  S(x,\mathrm{d} z),
  \]
  where we have set $S(x,\mathrm{d} z):= 2N(x,x+z)\dd  \sigma_x(z)$ and
  $\sigma_x(B):= \sigma(x+B)$ for every Borel set $B \subset \R^d$. 
  In the same way we can extend the drift
  term and we obtain an extended process $\overline{\X}$ into whole
  space $\R^d$.  Since we know that $\vX_t \in \partial
  D$ for all $t\geq 0$, it follows that the extended process $\overline{\X}$ will stay on the boundary
  $\partial D$ if we start it from the boundary and it coincides with
  $\vX$ there.
  
  For this extended process $\overline{\X}$ we can apply
  the result~\cite[Th{\'e}or{\`e}me 10]{LepeltierM}
  and we obtain
  \[
  \begin{split}
  &\EW_x \sum_{s \le \tau} \phi(\vX_{s-}, \vX_s)[\vX_{s-} \ne \vX_s] \\
  &= 2\EW_x \int_0^\tau 
  \int_{\R^d\setminus \{0\}} \phi(\vX_s, \vX_s + y)
  N(\vX_s, \vX_s + y) \dd  \sigma_{\vX_s}(y)  \dd  s.
  \end{split}
  \]
  The claim follows now in this special case by change of integration variable.

  The general case follows by scaling: Let us denote $\X^R_t :=
  R^{-1}\X_t$. This is a reflecting diffusion process corresponding to $\kappa^R$ on
  a domain $D^R$, where $D^R := R^{-1}D$ and $\kappa^R(x) :=
  R^{-2}\kappa(Rx)$. Since the diameter of $D^R$ is one, the claim holds
  for the boundary process $\vX^R$ of $\X^R$.
  
  Let $L^R$ denote the local time of $\X^R$ on the boundary $\partial
  D^R$. By definition, this is in Revuz correspondence with the surface
  measure $\sigma^R$ of the boundary $\partial D^R$. By using the Revuz
  correspondence \cite[Theorem 5.1.3]{Fukushimaetal} and change of
  variables, it follows that 
  \[
  L_t^R = RL_t.
  \]

  Therefore, the right-inverse $\tau^R$ of the local time $L^R$ has a
  scaling law 
  \[
  \tau^R(t) = \tau(R^{-1} t)
  \]
  which in turn implies that the boundary processes scale by the law
  \[
  \vX^R_t = R^{-1}\vX_{R^{-1}t}
  \]
  and that $\eta$ is an $\vX$-stopping time if and only if $R\eta$ is an
  $\vX^R$-stopping time.

  If we denote by $N^R$, the conormal derivative of the Poisson kernel of
  $\X^R$ and compute the scaling law, we find out that
  \[
  N^R(x,y) = R^{d-2} N(Rx,Ry).
  \]
  With all these scaling laws, we are now ready to prove the claim for $\vX$. Let
  $\phi^R(x,y) := \phi(x,y)$ and let $\eta$ be an $\vX$-stopping time. We have
  \[
  \EW_x \sum_{s \le \eta} \phi(\vX_s, \vX_{s-})[ \vX_s \ne \vX_{s-}]
  =
  \widehat{\EW}_{R^{-1}x} \sum_{s \le \eta R} \phi^R(\vX^R_s, \vX^R_{s-})[ \vX^R_s
  \ne \vX^R_{s-}],
  \]
  where $\widehat{\EW}_{R^{-1}x}$ denotes the expecation given $\vX^R_0
  = R^{-1} x$.
  By the first part of the proof, the right-hand side is
  \[
  2\widehat{\EW}_{R^{-1}x} \int_0^{\eta R} \int_{\partial {D^R}} \phi^R(\vX^R_s,
  y) N^R(\vX^R_s, y)\dd \sigma(y)\dd  s.
  \]
  With the change of variables $y' = R y$ and $s' = s R^{-1}$ and the
  scaling law $N^R(\vX^R_{Rs}, R^{-1} y) = R^{d-2} N(\vX_s, y)$, the
  claim follows.
\end{proof}
This result states that the pair $(\widetilde N, \mathrm {id})$ is the
L\'evy system (see~\cite[Definition A.3.7]{Fukushimaetal}) of the Hunt
process $\vX$ where $\widetilde N$ is the kernel on $(\partial D,
\mathcal B(\partial D))$ given by
\[
\widetilde N(x, B): = 2\int_B [x \notin B] N(x, y) \dd  \sigma(y)
\]
Since the PCAF $\mathrm {id}_t = t$ with respect to $\vX$ has the Revuz
measure $\sigma$, we see that $\frac12\widetilde N(x,B)\dd \sigma(x)$ is
the \emph{jumping measure} $J$ of the boundary process $\vX$
(see~\cite[Theorem 5.3.1]{Fukushimaetal}).

The same proof as in~\cite[Proposition 4.4]{Hsu2} shows that the random
set of jump times $\{ \vX_{s-} \ne \vX_s\}$ is a countable and dense set
and that there is a constant $c > 0$, depending only on the domain $D$, such that
after any given time $t \in \R_+$ there are always infinitely many jumps
of size at least $c$.

For the proof of Theorem~\ref{lemma:integrodiff} we need the following
auxiliary results.
\begin{lemma}\label{killed:diff}
  Suppose assumption (A2) holds. Then the transition
  density kernel $p_0$ of the killed diffusion $\X^0_t := [t < \tau_D]
  \X_t$ with lifetime $\tau_D$ is H\"older continuous with respect to
  $(t,x,y)$, is in $H_0^2(D)$ with respect to $x$ and $y$ and continuously
  differentiable with respect to $t$ as a Banach space valued map.
\end{lemma}
\begin{proof}
  The H\"older continuity follows from Theorem~\ref{thm:1} by Markov property, since
  \[
  p_0(t,x,y) = p(t,x,y)- \EW_x \big\{p(t-\tau_D, X_{\tau_D}, y) [t >
  \tau_D]\big\}.
  \]
  Moreover, we can carry out the same construction as in the proof of
  Lemma \ref{lem:exponential} only by replacing the space $H^1(D)$
  with $H^1_0(D)$. This yields an eigenfunction expansion of $p_0$,
  namely
  \[
  p_0(t,x,y) = \sum_{j=0}^\infty e^{-\mu_j t} \psi_j(x)\psi_j(y),
  \]
  which can be seen to be in $H_0^2(D)$ with respect to both variables $x$
  and $y$ by standard elliptic regularity theory, cf.~\cite{Grisvard}.
\end{proof}
\begin{lemma}
 \label{Poissonkernel}
  Suppose assumption (A2) holds, then 
 for every bounded and measurable $\phi$ on $\partial D$ and every bounded and measurable $\psi$ on
 $\R_+$ we have
 \[
 \EW_x \phi(\X_{\tau_D}) \psi(\tau_D) = -\int_{\partial D}
 \int_{\R_+}\phi(y)\psi(t)\partial_{\kappa\nu(y)} p_0(t,x,y)\dd  t \dd \sigma(y).
 \]
\end{lemma}
\begin{proof}
  The result follows by generalizing the results and proofs of \cite[Theorem
  A.3.2., Lemma A.3.3.]{AizenmanSimon} using the transition density kernel $p_0$ of the
  killed diffusion $\X^0$ instead of the transition density kernel of the
  killed Brownian motion. The proofs and the
  claims generalize in a straightforward manner by replacing the
  Brownian motion by the diffusion $\X$, the harmonic functions $h$ by
  the weak solutions of the conductivity equation $\nabla \cdot \kappa
  \nabla h = 0$ and the normal derivatives by conormal derivatives. The
  regularity we need for the proofs to go through follow from
  Lemma~\ref{killed:diff}.
\end{proof}
\begin{lemma}
  \label{singularity}
  Suppose assumption (A1) holds and the boundary $\partial D$ is of class $C^{1,1}$. For every $x \in \partial D$ the
  function $y \mapsto N(x,y) (\lvert y-x\rvert^2\wedge 1)$ is integrable with
  respect to the surface measure $\sigma$.
\end{lemma}
\begin{proof}
  When $\kappa \equiv 1$, the proof given in \cite{Hsu2} for $\kappa
  \equiv \frac12$ generalizes and
  gives the claim for kernel $N_1$ corresponding to $\kappa \equiv 1$. If
  we assume (A1), then the operator $\Lambda_\kappa - \Lambda_{\mathbf
  1}$ is a smoothing operator which follows by the standard elliptic
  regularity, cf.~\cite{HankeHR}. This implies that the kernels $N_1$
  and $N$ have the same leading
  singularities and the claim follows from the estimate for $N_1$.
\end{proof}
\begin{proof}[Proof of Theorem~\ref{lemma:integrodiff}]
Suppose $\phi \in C^2(\partial D) \cap H^{3/2}(\partial D)$. The solution
of the Dirichlet problem is by Lemma~\ref{Poissonkernel}
\[
u(x) = \EW_x \phi(\X_{\tau_D}) = \int_{\partial D} 
 \int_{\R_+}\phi(y)\partial_{\kappa\nu(y)} p_0(t,x,y)\dd  t\dd \sigma(y) 
 =: \mathcal K\phi(x).
\]
Therefore, the Dirichlet-to-Neumann map $\Lambda_\kappa$ maps $\phi$ to 
\begin{equation}
  \label{equ:DN1}
  \Lambda_\kappa \phi(x) = \partial_{\kappa\nu(x)} \mathcal K\phi(x).
\end{equation}
Let us extend $\phi$ and its first order tangential derivative $V :=
\nabla_T \phi$ into
the neighborhood of the boundary as constant along the conormal
directions. We will denote the extensions $\widetilde \phi$ and $\widetilde
V$, respectively. We will compute the conormal derivative of the function
\[
w: = \mathcal K \phi - \widetilde \phi \mathcal K \mathbf 1 -
\sum_{j=1}^d \widetilde
{V_j} \big(\mathcal K W_j - \widetilde W_j \mathcal K \mathbf 1\big),
\]
where $\{W\}$ is a vector field on the boundary defined by
$W (y): = y_T$ as the projection to the tangent plane going through the
point $y$.
By construction the conormal derivative commutes with
multiplication by the extended functions and vector fields. Therefore,
\[
\partial_{\kappa\nu} w = \Lambda_\kappa \phi - \phi \Lambda_\kappa \mathbf 1 - 
\sum_{j=1}^d V_j \Lambda_{\kappa} W_j
= \Lambda_\kappa \phi - V \cdot \Lambda_{\kappa} W,
\]
where $\Lambda_\kappa \mathbf 1 = 0$ since $u(x) = 1$ in $\overline D$
is the unique solution to the Dirichlet problem and therefore, the
conormal derivative vanishes identically.
We can compute the left-hand side in a different way, namely 
\[
\nabla w(x) = \nabla_x \int_{\partial D} \big(\phi(y)-\widetilde \phi(x)-
\widetilde V(x)\cdot (\widetilde W(y) - \widetilde W(x))\big) K(x,y) \dd \sigma(y)
\]
for almost every $x$ in a neighborhood of boundary.
By Lemma~\ref{singularity} we can use the Dominated Convergence Theorem
to take the differentiation inside the integration and we obtain thus
\[
\partial_{\kappa\nu} w(x) = \int_{\partial D} \big(\phi(y)- \phi(x)-
\nabla_T\phi(x)\cdot (y - x)\big) N(x,y) \dd  \sigma(y) = A_0\phi(x)
\]
for almost every $x \in \partial D$.
\end{proof}

We have demonstrated that the boundary process $\vX$ is intimately related
with the Dirichlet-to-Neumann map $\Lambda_\kappa$. In order to solve the probabilistic
inverse problem of recovering the process $\X_\kappa$ inside the domain,
we should provide the \emph{excursions} between the points $\vX_{t-}$
and $\vX_t$ at the jumps so that the L\'evy system of the
boundary process $\vX$ and \emph{excursion law} of these
excursions would be consistent, cf.~\cite{Hsu2}. These excursions could
be regarded as point processes on the space of continuous functions that
start at the boundary, stay inside the domain and stop at the boundary.

The point process of excursions for the case $\kappa \equiv 1$ was
defined in~\cite{Hsu2} and under a certain consistency assumption it was shown
that the interior RBM $\X$ can be reconstructed from its point
excursions and the boundary process. However, the consistency assumption
was derived by using RBM to begin with and as it is noted in~\cite{Hsu2}, there might 
be other consistency assumptions leading to other
constructions. Showing that there are no other consistent constructions
is equivalent to the probabilistic inverse problem.

However, we will not attempt to analyze the excursion processes for
conductivities $\kappa$ in this paper and therefore, we will leave the
analysis of the probabilistic inverse problem for future work.

\section{Conclusion}\label{section6}
We have obtained probabilistic interpretations of both the direct as well as the 
inverse problem of electrical impedance tomography. Using the theory of 
symmetric Dirichlet spaces we have derived Feynman-Kac type representation 
formulae generalizing the probabilistic representations from 
\cite{Bass, Pardoux, Papanicolaou}. These formulae are potentially relevant 
in statistical inversion theory as well as stochastic numerics of 
problems involving random, rapidly oscillating coefficients. Furthermore we have given 
a probabilistic formulation of Calder\'{o}n's inverse conductivity problem, 
generalizing results from \cite{Hsu2}, which may yield 
a novel perspective and a probabilistic set of tools when it comes to studying 
the open question of unique determinability of merely measurable conductivities 
for $d\geq 3$.

\section*{Acknowledgments} 
The research of M.~Simon was supported by the Deutsche
Forschungsgemeinschaft (DFG) under grant HA 2121/8 -1 583067.
This work is part of M. Simon's Ph.D thesis, who would like to express his gratitude to his 
advisor Prof.~M.~Hanke for his guidance and continuous support. He
would also like to thank Prof.~L.~P\"{a}iv\"{a}rinta for the kind
invitation to the Department of Mathematics and Statistics at the
University of Helsinki, where part of the work was carried out.
The research of P.~Piiroinen was supported by Academy of Finland (AF) under Finnish
Centre of Excellence in Inverse Problems Research 2012--2017,
decision number 250215. He has also been supported by an AF project,
decision number 141075.

\end{document}